\input ./style/arxiv-general.cfg
\documentclass[aop,MSNbibl,seceqn,dvips]{arximspdf}
\makeatletter
   \@ifpackageloaded{graphicx}{}{\usepackage{graphicx}}
\makeatother
\usepackage{tikz}

\doi{10.1214/14-AOP938} 
\volume{43}
\issue{5}
\pubyear{2015}
\firstpage{2374}
\lastpage{2404}
\docsubty{FLA}

\makeatletter
\def\adjustlimits{}
\newcommand{\eqref}[1]{(\ref{#1})}
\newcommand{\Hbb}{\mathbb{H}}
\newcommand{\Sbb}{\mathbb{S}}
\newcommand{\boD}{\mathcal{D}}
\newcommand{\boH}{\mathcal{H}}
\newcommand{\N}{\mathbb{N}}
\newcommand{\R}{\mathbb{R}}

\newcommand{\st}{\dvtx}
\newcommand{\sumstar}{{\sum}^{'}}

\newcommand{\Leb}{\operatorname{Leb}}
\newcommand{\Card}{\operatorname{Card}}

\renewcommand{\epsilon}{\varepsilon}
\renewcommand{\phi}{\varphi}

\newtheorem{thmm}{Theorem}[section]
\newtheorem{prop}[thmm]{Proposition}
\newproclaim{definition}[thmm]{Definition}
\newtheorem{lem}[thmm]{Lemma}

\newtheorem{step}{Step}

\newproclaim{ex}{Example}
\newproclaim{rmk}[thmm]{Remark}
\makeatother

\begin{document}
\begin{frontmatter}

\title{Martin boundary of random walks with unbounded jumps in
hyperbolic groups}
\runtitle{Martin boundary of random walks with unbounded jumps}

\begin{aug}
\author[A]{\fnms{S\'ebastien}~\snm{Gou\"ezel}\corref{}\ead[label=e1]{sebastien.gouezel@univ-rennes1.fr}}
\runauthor{S. Gou\"ezel}
\affiliation{IRMAR, Universit\'e de Rennes 1}
\address[A]{IRMAR\\
CNRS UMR 6625\\
Universit\'e de Rennes 1\\
35042 Rennes\\
France\\
\printead{e1}}
\end{aug}

\received{\smonth{9} \syear{2013}}

%
\begin{abstract}
Given a probability measure on a finitely generated group, its Martin
boundary is a natural way to compactify the group using the Green
function of the corresponding random walk. For finitely supported
measures in hyperbolic groups, it is known since the work of Ancona
and Gou\"ezel--Lalley that the Martin boundary coincides with the
geometric boundary. The goal of this paper is to weaken the finite
support assumption. We first show that, in any nonamenable group,
there exist probability measures with exponential tails giving rise
to pathological Martin boundaries. Then, for probability measures
with superexponential tails in hyperbolic groups, we show that the
Martin boundary coincides with the geometric boundary by extending
Ancona's inequalities. We also deduce asymptotics of transition
probabilities for symmetric measures with superexponential tails.
\end{abstract}

%
%
\end{frontmatter}

\section{Introduction}\label{sec1}

Consider a probability measure $\mu$ on a finitely generated group
$\Gamma$, whose support generates $\Gamma$ as a semigroup (we say
that $\mu$ is \emph{admissible}). The Green function associated to
$\mu$ is $G_{\mu}(x,y)=G(x,y) = \sum_{n=0}^\infty\mu^n(x^{-1}y)$.
The Green function is defined so that the random walk with transition
probabilities $p(a,b) = \mu(a^{-1}b)$ starting from $x$ spends an
average time $G(x,y)$ at $y$. We will always assume that this sum is
finite (i.e., the random walk is transient). The function $G$
contains a lot of information about the transition probabilities and
the asymptotic properties of the random walk. Moreover, it is at the
heart of the potential theory of $\mu$, making it possible to
describe all positive harmonic functions through the notion of
\emph{Martin boundary}.

The Martin boundary $\partial_\mu\Gamma$ is defined as follows: a
sequence of points $y_n\in\Gamma$ going to infinity converges in
$\Gamma\cup\partial_\mu\Gamma$ if and only if, for all $z$, the
sequence $K_{y_n}(z)=G(z,y_n)/G(e,y_n)$ converges, where $e$ denotes
the identity of the group. One can associate to any $\xi\in
\partial_\mu\Gamma$ the corresponding Martin kernel $K_\xi(z) =
\lim
K_{y_n}(z)$. This function is superharmonic (i.e., if $P_\mu$ denotes
the Markov operator associated to $\mu$, then $P_\mu K_\xi\leq
K_\xi$), and any positive superharmonic function on $\Gamma$ can be
decomposed as an integral of the kernels $K_\xi$ with respect to some
finite measure on $\Gamma\cup
\partial_\mu\Gamma$ (the decomposition is unique if one requires that
the measure is supported on $\Gamma$ and on the minimal part of the
Martin boundary, made of those $\xi$ whose kernel $K_\xi$ is harmonic
and minimal among positive harmonic functions). See, for
instance, \cite{dynkin_martin,sawyer_martin,woess}.

Describing concretely the Martin boundary in specific examples is
difficult, especially in nonamenable situations. A landmark result
in this direction is a theorem by Ancona~\cite{ancona2} showing that,
for finitely supported probability measures in (nonelementary)
hyperbolic groups, the Martin boundary coincides with the geometric
boundary of the group. His result is not restricted to probability
measures: the Green function and the Martin boundary can be defined
for any finite measure $\mu$, and Ancona's result is true for any
measure $\mu$ such that $r\mu$ has a finite Green function for some
$r>1$ [we will say that such a $\mu$ has the property $\mathrm
{Anc}_*$, since
this property is called $(*)$ in Ancona's paper]. Ancona's proof is
based on an inequality saying that, in hyperbolic groups, the Green
function of a measure with finite support and property $\mathrm{Anc}_*$ is
essentially multiplicative along geodesics: there exists a constant
$C$ such that, for any $x,y,z$ on a geodesic of the group (in this
order), one has
%
\begin{equation}
\label{eq:ancona} C^{-1} G(x,y)G(y,z) \leq G(x,z) \leq CG(x,y)G(y,z).
\end{equation}
While the first inequality is true for any random walk in any group,
the second one is highly nontrivial. It is used by Ancona to show
that the Martin boundary coincides with the geometric boundary. It
also plays an important role in the article~\cite{BHM_2} by Blach\`ere,
Ha\"{\i}ssinsky and Mathieu: they prove that this inequality is necessary
and sufficient so that a natural distance associated to the random
walk, the \emph{Green distance}, is hyperbolic (and they prove
several properties of the harmonic measure at infinity under this
condition). It is also instrumental in the
articles~\cite{gouezel_lalley,gouezel_higherancona} by Gou\"ezel and
Lalley, where the asymptotics of transition probabilities in
hyperbolic groups are determined (note that the authors need to
extend Ancona inequalities to some measures that do not satisfy
$\mathrm{Anc}_*$). All those results rely on the finiteness of the
support of the
measure $\mu$.

Our goal in this article is to see to what extent the previous
results can be extended to measures with infinite support. The tails
of the measure, that is, the speed at which $\mu(B(e,n)^c)$ tends to $0$
[where $B(e,n)^c$ denotes the complement of the ball centered at $e$
of radius $n$, for some word distance in the group] will play an
important role in the results. We will say that a measure has
exponential tails if there exists $K>1$ such that, for large enough
$n$, $\mu(B(e,n)^c) \leq K^{-n}$. We will say that $\mu$ has
superexponential tails if this condition is true for all $K>1$.
Equivalently, $\mu$ has exponential tails if, for some $\delta>0$,
the sum $\sum_{g\in\Gamma} e^{\delta\vert g\vert}\mu(g)$ is finite
(where $\vert g\vert$ is the distance from $e$ to $g$ in a word metric),
and $\mu$ has superexponential tails if this sum is finite for all
$\delta>0$.

Our first result shows that one cannot expect a reasonable
description of the Martin boundary if one only demands an exponential
decay of the tails:

\begin{thmm}
\label{thmm:pathological}
Consider a nonamenable finitely generated group $\Gamma$, and a
sequence $y_n$ going to infinity in $\Gamma$. There exists an
admissible symmetric probability measure $\mu$ on $\Gamma$, with
exponential tails, such that $y_n$ does not converge in the Martin
boundary $\partial_\mu\Gamma$.
\end{thmm}

This implies in particular that there exist uncountably many possible
different Martin boundaries for measures with exponential tails, by a
standard diagonal argument.

If the tails have a better behavior (i.e., if they are
superexponential), we can extend Ancona's results:

\begin{thmm}
\label{thmm:old_ancona}
In a nonelementary hyperbolic group $\Gamma$, consider an admissible
measure satisfying $\mathrm{Anc}_*$, with superexponential tails. Then it
satisfies Ancona inequalities~\eqref{eq:ancona}. In particular, its
Martin boundary coincides with the geometric boundary of the group.
\end{thmm}

It follows that all the results of~\cite{BHM_2} describing the
geometry of the harmonic measure (and in particular its pointwise
dimension), originally obtained for finitely supported measures,
still hold for measures with superexponential tails.

As we explained before, the results of~\cite
{gouezel_lalley,gouezel_higherancona} require Ancona inequalities for
measures that
do not satisfy $\mathrm{Anc}_*$. We extend their results to measures with
superexponential tails.

\begin{thmm}
\label{thmm:new_ancona}
In a nonelementary hyperbolic group $\Gamma$, consider an admissible
measure $\mu$ with superexponential tails and finite Green function.
Assume that one of the following conditions is satisfied:
\begin{enumerate}
\item The measure $\mu$ is symmetric.
\item The group $\Gamma$ is a free group on finitely many
generators.
\item The group $\Gamma$ is a cocompact lattice of ${\mathrm
{PSL}(2,\R)}$.
\end{enumerate}
Then $\mu$ satisfies Ancona inequalities~\eqref{eq:ancona}. In
particular, its Martin boundary coincides with the geometric boundary
of the group.
\end{thmm}

It is likely that the above conditions ($\mu$ symmetric or $\Gamma$
planar) are not necessary for this theorem, but this is unknown even
in the case of a finitely supported $\mu$. The above conditions are
precisely those that are used in~\cite
{gouezel_lalley,gouezel_higherancona} to obtain (for finitely supported
measures)
Ancona inequalities and a description of the Martin boundary.

The motivation for the results of~\cite
{gouezel_lalley,gouezel_higherancona} was to obtain asymptotics of transition
probabilities for random walks. We deduce the corresponding statement
in our setting.

\begin{thmm}
\label{thmm:local}
In a nonelementary hyperbolic group $\Gamma$, consider a symmetric
admissible probability measure $\mu$ with superexponential tails.
Denote by $R>1$ the inverse of the spectral radius of the
corresponding random walk. For any $x, y\in\Gamma$, there exists
$C(x,y)>0$ such that the transition probabilities $p^n(x,y)$ of the
random walk at time $n$ satisfy
\[
p^n(x,y) \sim C(x,y)R^{-n} n^{-3/2}
\]
if the walk is aperiodic. If the walk is periodic, this asymptotics
holds for even (resp., odd) $n$ if the distance from $x$ to $y$ is
even (resp., odd).
\end{thmm}

This result is new even for random walks on free groups. Note that,
even in the finitely supported case, the proof requires the symmetry
of the measure since the very end of the argument relies on spectral
properties of the Markov operator.

The paper is organized as follows. In
Section~\ref{sec:green_function}, we recall basic properties of the
Green function. Section~\ref{sec:pathological} is devoted to the
construction of pathological Green functions for measures with
exponential tails, proving in particular
Theorem~\ref{thmm:pathological}. The main idea of the construction is
that, even with exponential tails, one can ensure that the most
likely way to reach some point is by doing a direct jump. This makes
it possible to prescribe very precisely the asymptotics of the Green
function. Finally, Section~\ref{sec:positive} is devoted to the
positive results in hyperbolic groups, for measures with
superexponential tails. Ancona's arguments to get his inequality rely
on a subtle induction that does not seem generalizable to the
infinite support situation. We will rather use a lemma
of~\cite{gouezel_lalley} (see Lemma~\ref{lem:GL} below) showing that
some upper bounds on relative Green functions imply Ancona
inequalities. Such upper bounds are more manageable, and can be
proved for infinitely supported measures as we will show.

\section{The Green function}
\label{sec:green_function}

Consider a finite admissible measure $\mu$ on a finitely generated
group $\Gamma$. We will always assume that its Green function
$G(x,y)=\sum\mu^n(x^{-1}y)$ is finite for some $x,y$ (and, therefore,
for all $x,y$ by admissibility). Denote by $P_\mu$ the operator
associated to $\mu$, given by $P_\mu f(x) = \sum\mu(x^{-1}y) f(y)$---when
$\mu$ is a probability measure, this is simply the Markov
operator associated to the corresponding random walk. Even when $\mu$
is not a probability measure, we will use probabilistic notation
such as $p^n(x,y) = \mu^n(x^{-1}y)$, and think of $G(x,y) = \sum
(P_\mu^n \delta_y)(x)$ as an average time spent at $y$ if one starts
from $x$.

The Green function can also be formulated in terms of paths. Let
$\tau=(x, x_1,\ldots,x_{n-1}, y)$ be a path of length $n$ from $x$
to $y$, we define its $\mu$-weight (or simply weight)
$\pi_\mu(\tau)=\pi(\tau)$ by
\[
\pi(\tau) = \prod_{i=0}^{n-1}
p(x_i, x_{i+1}),
\]
where $x_0=x$ and $x_n=y$ by convention, and we write
$p(a,b)=\mu(a^{-1}b)$. We think of $\pi(\tau)$ as the ``probability''
to follow the path $\tau$. By definition, $G(x,y)=\sum\pi(\gamma)$,
where the sum is over all paths from $x$ to $y$.

If $\Omega$ is a subset of $\Gamma$, one defines the restricted Green
function $G(x,y;\Omega)$ as $\sum\pi(\gamma)$ where the sum is over
all paths $\gamma=(x,x_1,\ldots, x_{n-1},y)$ such that $x_i\in
\Omega$ for $1\leq i\leq n-1$. If $A$ is a subset of $\Gamma$ and
$x,y\notin A$, one has
\begin{eqnarray*}
G(x,y) & =& G\bigl(x,y; A^c\bigr) + \sum
_{a\in A} G\bigl(x,a; A^c\bigr) G(a,y)
\\
& =& G\bigl(x,y; A^c\bigr) + \sum_{a\in A}
G(x,a) G\bigl(a, y; A^c\bigr),
\end{eqnarray*}
where $A^c$ denotes the complement of $A$. Indeed, the first (resp.,
second) formula is proved by splitting a path from $x$ to $y$
according to its first (resp., last) visit to $A$ if it exists, the
remaining trajectories giving the contribution $G(x,y; A^c)$. If all
trajectories from $x$ to $y$ have to go through $A$, this
contribution vanishes. This is used crucially in the usual arguments
for finitely supported measures, where one uses wide enough
``barriers'' $A$ between $x$ and $y$, that any trajectory from $x$ to
$y$ has to visit. In the infinite support situation, the contribution
$G(x,y; A^c)$ will always be present.

More generally, if $\Omega$ is a subset of $\Gamma$ containing $x$
and $y$, the above formula holds restricted to $\Omega$, that is,
%
\begin{eqnarray}
\label{eq:decompose_G} G(x,y; \Omega) & =& G\bigl(x,y; \Omega\cap A^c\bigr)+
\sum_{a\in A\cap\Omega} G\bigl(x,a; \Omega \cap A^c
\bigr) G(a,y; \Omega)
\nonumber
\\[-8pt]
\\[-8pt]
\nonumber
& =& G\bigl(x,y; \Omega\cap A^c\bigr)+\sum
_{a\in A\cap\Omega} G(x,a;\Omega) G\bigl(a, y; \Omega\cap A^c
\bigr).
\end{eqnarray}

Let $d$ be a word distance on $\Gamma$ coming from a finite symmetric
generating set. If $x$ and $y$ are at distance $d$, there is a path
from $x$ to $y$ with weight bounded from below by $C^{-d}$, and
staying close to a geodesic segment from $x$ to $y$. We deduce that,
for any $z$,
%
\begin{equation}
\label{eq_harnack} C^{-d(x,y)} \leq G(x,z)/G(y,z) \leq C^{d(x,y)},
\end{equation}
and similar inequalities hold for the Green function restricted to
any set containing a fixed size neighborhood of a geodesic segment
from $x$ to $y$. These inequalities are called Harnack inequalities.

The first visit Green function is $F(x,y)=G(x,y; \{y\}^c)$. It only
takes into account the first visits to $y$. When $\mu$ is a
probability measure, $F(x,y)$ is the probability to reach $y$
starting from $x$. One has $G(x,y)=F(x,y) G(y,y)=F(x,y) G(e,e)$.
Moreover, $F(x,y) G(y,z) \leq G(x,z)$ (since the concatenation of a
path from $x$ to $y$ with a path from $y$ to $z$ gives a path from
$x$ to $z$). Hence,
%
\begin{equation}
\label{eq:super_mult} G(x,y)G(y,z) \leq G(e,e) G(x,z).
\end{equation}
This shows that the left inequality in~\eqref{eq:ancona} is always
true.

\section{Pathological constructions in nonamenable groups}
\label{sec:pathological}

Let $\Gamma$ be a finitely generated nonamenable group. In this
section, we construct admissible symmetric probability measures with
exponential tails that behave in a pathological way regarding their
Green functions and Martin boundaries.

The basic idea is the following. We start from a symmetric
probability measure $\nu$ supported by a finite generating set of
$\Gamma$, and we add Dirac masses, with a very small mass but
supported far away from the identity. If we adjust carefully the
weights, the way to reach some far away points with highest
probability is to jump directly onto them (possibly with some short
jumps), since an accumulation of small jumps has a lower probability
that one single big jump. In this way, we will prescribe the behavior
of the Green function at different scales.

This type of behavior is reminiscent of L\'evy processes on $\R$: when
such a process is large, this is typically due to one single large
jump, the sum of the other jumps being negligible. We are
constructing a kind of L\'evy process on $\Gamma$, but with exponential
tails. The reason behind this counterintuitive phenomenon (in $\R$,
L\'evy processes need to have heavy tails) is that exponentially small
tails can still dominate the diffusive behavior since the diffusion
is also exponentially small in nonamenable groups.

The precise construction is as follows. Let $\rho<1$ be the spectral
radius of the random walk given by $\nu$. It is also the norm of the
associated Markov operator $P_\nu$ since $\nu$ is symmetric. Let us
fix a decreasing sequence $r_i$ (the exponential weights) with
$e^{r_0} \rho<1$ and $\lim r_i = r > 0$. Let us also fix a sequence
$n_i$ tending very quickly to infinity, and a symmetric measure
$\mu_i$ supported on the ball $B(e,n_i)$. Let
\[
\mu= \nu+ \sum e^{-r_i n_i} \mu_i \quad\mbox{and}\quad
\mu'=\mu/ \mu (\Gamma).
\]
The probability measure $\mu'$ is symmetric, and has exponential
tails of order $r$. We will see that we can prescribe the behavior of
its Green function. Since most interesting things happen with one
jump, we may equivalently work with $\mu'$ or $\mu$. It will be more
convenient to formulate the estimates for $\mu$.

The fact that $r_i$ is strictly decreasing is a central point of the
construction. Roughly speaking, if one uses only the measures $\mu_i$
with $i\leq I$, then a jump of size $n\leq n_I$ is made with
probability at most $e^{-r_I n}$. This implies that a point at a
large distance $n$ of $e$ will be reached with probability roughly
$e^{-r_I n}$. Let us take $n=n_{I+1}$, and $x$ a point in the support
of $\mu_{I+1}$. It can be reached by a direct jump, with probability
of the order of $e^{-r_{I+1} n}$, which is much bigger than $e^{-r_I
n}$ since $r_{I+1}< r_I$. Hence, direct jumps are more likely than a
combination of small jumps, as desired.

The rigorous version of this argument is slightly more complicated:
using the measures $\mu_i$ with $i\leq I$, one can in fact reach a
point at distance $n$ with a probability at most $C(s) e^{-s n}$ for
any $s<r_I$. Hence, we need to introduce another sequence: we fix
once and for all $s_{i+1} \in(r_{i+1}, r_i)$ (we also require that
$s_{i+1}<2r_{i+1}$ for technical reasons). In the following, we will
always assume that $n_i$ grows quickly enough so that
%
\begin{equation}
\label{eq:main_cond0} r_{i+1} n_{i+1} \geq s_i
n_i \geq r_i n_i + i + 1
\end{equation}
and
%
\begin{equation}
\label{eq:main_cond} \frac{1}{1 - \rho e^{r_0}} \sum e^{-(r_i-s_{i+1}) n_i} \leq
\frac{1}{2}.
\end{equation}
Since $r_i-s_{i+1}>0$, this can easily be guaranteed. From this point
on, the letter $C$ will denote a constant that can vary from one line
to the other, but does not depend on the choices we have made
provided the conditions~\eqref{eq:main_cond0}
and~\eqref{eq:main_cond} are satisfied.

Let us estimate the Green function $G(e,x)$ associated to $\mu$. This
is the sum of the weights of paths from $e$ to $x$. We will group
together those paths corresponding to the same sequence of measures
$\nu$ or $\mu_i$. This is most conveniently done in terms of Markov
operators as follows. We will write $P=P_\nu$ and $P_i=P_{\mu_i}$ for
the operators associated, respectively, to $\nu$ and $\mu_i$. They
satisfy $P_\mu=P+ \sum e^{-r_i n_i} P_i$. Developing $P_\mu^n$ and
grouping together the successive occurrences of $P$, we get
\begin{eqnarray*}
G(e,x) & =& \sum_n \bigl\langle
P_\mu^n \delta_x, \delta_e
\bigr\rangle
\\
& = &\sum_{\ell=0}^\infty\sum
_{a_0,i_1,a_1,\ldots, i_\ell, a_\ell} \bigl\langle P^{a_0} e^{-r_{i_1} n_{i_1} }
P_{i_1} P^{a_1} \cdots P^{a_{\ell-1}} e^{-r_{i_\ell} n_{i_\ell} }
P_{i_\ell}P^{a_\ell} \delta_x, \delta _e
\bigr\rangle.
\end{eqnarray*}
Each term in the double sum corresponds to the weight of several
trajectories. We will say that the associated sequence $t=(a_0, i_1,
a_1,\ldots, a_\ell)$ is a \emph{template} for this set of
trajectories. The norm of $P^a$ on $\ell^2(\Gamma)$ is bounded by
$\rho^a$, and the norm of $P_i$ is at most $1$. Hence, the sum of the
weights of trajectories in a template $t$ is bounded by its weight
$\pi(t)$ defined by
\[
\pi(t) = \rho^{a_0+\cdots+a_\ell}e^{-r_{i_1} n_{i_1}}\cdots e^{-r_{i_\ell} n_{i_\ell}}.
\]
Summing over the templates, we obtain
%
\begin{equation}
\label{eq:borne_basique} G(e,x) \leq \sumstar\pi(t),
\end{equation}
where the notation $\sum'$ indicates that we can remove from the sum
all those templates that give a vanishing contribution to $G(e,x)$,
that is, those for which no trajectory can go from $e$ to $x$.

It is not clear that the Green function of $\mu$ is well defined,
since $\mu$ is not a probability measure. We can
use~\eqref{eq:borne_basique} to show its finiteness, uniformly in
$x$. We have
%
\begin{equation}
\label{eq:geom} \sum_t \pi(t) \leq\sum
_\ell \Biggl( \sum_{a=0}^\infty
\rho ^a \Biggr)^{\ell+1} \biggl(\sum
_i e^{-r_i n_i} \biggr)^\ell.
\end{equation}
The sum over $\ell$ is a geometric series. It is finite if its
general term is $<1$, that is, $\frac{1}{1-\rho}\sum e^{-r_i n_i}<1$.
This is a consequence of (stronger)
condition~\eqref{eq:main_cond}. As $G(e,x) \leq\sum\pi(t)$, this
shows that $G(e,x)$ is well-defined and uniformly bounded.

We need more notation regarding templates. Given a template
$t=(a_0, i_1,\break  a_1,\ldots,a_\ell)$, define its length $\vert t\vert =
\sum a_k + \sum n_{i_k}$: any trajectory in the template ends at a
point at distance at most $\vert t\vert$ of the origin. Let also $\max t
= \sup i_k$ give the size of the biggest jump in $t$. We will write
$t_1 \cdot t_2$ for the concatenation of two templates $t_1$ and
$t_2$. It satisfies $\pi(t_1 \cdot t_2) = \pi(t_1) \pi(t_2)$.

The crucial estimates for template weights are the following.

\begin{lem}
\label{lem:ineg_i}
For every integers $i$ and $n$,
%
\begin{equation}
\label{eq:somme_large} \sum_{\max t \geq i} \pi(t) \leq C
e^{-r_i n_i}
\end{equation}
and
%
\begin{equation}
\label{eq:somme_petite} \sum_{\max t < i, \vert t\vert \geq n} \pi(t) \leq C
e^{-s_i n}.
\end{equation}
As a consequence, for every $i\in\N$ and for every $z\in\Gamma$,
%
\begin{equation}
\label{eq:ineg_i} G(e, z) \leq C e^{-r_i n_i} + C e^{-s_i \vert z\vert}.
\end{equation}
\end{lem}

Inequality~\eqref{eq:somme_large} controls what happens when
there is at least one big jump, while~\eqref{eq:somme_petite}
controls the combination of several small jumps. The last
inequality~\eqref{eq:ineg_i} is a consequence of the other two. Note
that, if $\vert z\vert$ is comparable to~$n_i$, then the second term
in~\eqref{eq:ineg_i} is negligible compared to the first one since
$s_i n_i - r_i n_i\to+\infty$ by~\eqref{eq:main_cond0}. This shows
rigorously that the most efficient way to visit $z$ is to do one big
jump rather than many small jumps, as we already explained
informally.

\begin{pf*}{Proof of Lemma \ref{lem:ineg_i}}
Let us first show~\eqref{eq:somme_large}. A template $t$ with $\max
t\geq i$ can be decomposed as $t=t_1 \cdot(j) \cdot t_2$ where $t_1$
and $t_2$ are shorter templates and $j$ corresponds to a jump of size
$n_j\geq n_i$. Therefore,
\[
\sum_{\max t \geq i} \pi(t) \leq \biggl(\sum
_{t_1} \pi(t_1) \biggr) \Biggl(\sum
_{j=i}^\infty e^{-r_j
n_j} \Biggr) \biggl(\sum
_{t_2} \pi(t_2) \biggr).
\]
The first sum and the last sum are finite by~\eqref{eq:geom}. The
middle one is bounded by $C e^{-r_i n_i }$ thanks
to~\eqref{eq:main_cond0}. This proves~\eqref{eq:somme_large}.

Let us now show~\eqref{eq:somme_petite}. Writing $t=(a_0,i_1,\ldots,
a_\ell)$, the corresponding sum is
\[
\sum_{\max t < i,\vert t\vert\geq n} e^{-s_i(a_0+\cdots+a_\ell+
n_{i_1}+\cdots+ n_{i_\ell})} \bigl(\rho
e^{s_i}\bigr)^{a_0+\cdots+ a_\ell} e^{-(r_{i_1}-s_i) n_{i_1}} \cdots e^{-(r_{i_\ell} -s_i) n_{i_\ell}}.
\]
The first factor is $e^{-s_i \vert t\vert} \leq e^{-s_i n}$. This yields
a bound
\begin{eqnarray*}
&&e^{-s_i n} \sum_{\max t < i} \bigl(\rho
e^{s_i}\bigr)^{a_0+\cdots+ a_\ell} e^{-(r_{i_1}-s_i) n_{i_1}} \cdots e^{-(r_{i_\ell} -s_i) n_{i_\ell}}
\\
&&\qquad= e^{-s_i n} \sum_\ell \Biggl(\sum
_{a=0}^\infty\bigl(\rho e^{s_i}
\bigr)^a \Biggr)^{\ell+ 1} \Biggl(\sum
_{j=0}^{i-1} e^{-(r_j -s_i) n_j} \Biggr)^\ell.
\end{eqnarray*}
This is again a geometric series. Let us bound $e^{s_i}$ with
$e^{r_0}$ in the first factor, and $e^{-(r_j -s_i) n_j}$ with
$e^{-(r_j - s_{j+1}) n_j}$ in the second factor. We get that the
general term of this geometric series is bounded by
\[
\frac{1}{1- \rho e^{r_0}} \sum_{j\geq0} e^{-(r_j - s_{j+1}) n_j}.
\]
Condition~\eqref{eq:main_cond} guarantees that this is $\leq1/2$.
Hence, the geometric series is uniformly bounded, yielding a bound $C
e^{-s_i n}$. This proves~\eqref{eq:somme_petite}.

Let us finally prove~\eqref{eq:ineg_i}
using~\eqref{eq:borne_basique}. To go from $e$ to $z$, the templates
with $\max t \geq i$ give an overall contribution at most $C e^{-r_i
n_i}$, by~\eqref{eq:somme_large}. On the other hand, if $\max t < i$,
then it is possible to reach $z$ using a trajectory in the template
only if $\vert t\vert \geq\vert z\vert$. By~\eqref
{eq:somme_petite}, those
terms contribute at most $C e^{-s_i \vert z\vert}$.
\end{pf*}

This lemma implies that, in general, there is no Ancona
inequality~\eqref{eq:ancona} in nonamenable groups, for measures
with exponential tails.

\begin{prop}
\label{prop:pas_ancona}
Let $\Gamma$ be a finitely generated nonamenable group. There exists
on $\Gamma$ an admissible symmetric probability measure $\mu'$ with
exponential tails whose Green function $G'=G_{\mu'}$ does not satisfy
Ancona inequalities: there is no constant $C$ such that $G'(x,z) \leq
C G'(x,y)G'(y,z)$ for any $x,y,z\in\Gamma$ on a geodesic in this
order.
\end{prop}

\begin{pf}
We use the previous construction, with $\mu_i = (\delta_{z_i} +
\delta_{z_i^{-1}})/2$ where $z_i$ is a point at distance $n_i$ of
$e$. We will assume that $n_i$ is even, and we will denote by $y_i$
the midpoint of a geodesic segment from $e$ to $z_i$. We will show
that
%
\begin{equation}
\label{eq:xi} G'(e, z_i) \geq C e^{-r_i n_i}
\end{equation}
and
%
\begin{equation}
\label{eq:yi} G'(e,z) \leq C e^{-s_i n_i/2}
\end{equation}
for any $z$ with $d(e,z)=n_i/2$. Hence, $G'(e,y_i) G'(y_i, z_i) \leq
C^2 e^{-s_i n_i}=\break o(G'(e,z_i))$, contradicting any Ancona inequality.

Inequality~\eqref{eq:xi} is obvious since the Green function is
bounded from below by the contribution of single jumps: $G'(e,z_i)
\geq\mu'(z_i) = \mu(\Gamma)^{-1} e^{-r_i n_i}/2$.

As $G'\leq G$, inequality~\eqref{eq:yi} follows
from~\eqref{eq:ineg_i} since $\vert z\vert=n_i/2$. [The first term
in~\eqref{eq:ineg_i} is dominated by the second term since we have
requested that $s_i < 2 r_i$.]
\end{pf}

We now turn to the proof of Theorem~\ref{thmm:pathological}. Starting
from a sequence $y_n$ going to infinity, we wish to construct the
measures $\mu$ and $\mu'$ (using the above construction) so that
$G'=G_{\mu'}$ is such that, for some point $z$, the sequence
$G'(z,y_n)/G'(e,y_n)$ does not converge. We will write $G'=G_{\mu'}$
and $G=G_{\mu}$.

We need to fix an additional sequence $s'_i \in(r_i, s_i)$, for
instance the middle of this interval, to get some additional freedom.
Taking a subsequence of $y_n$, we can assume that
%
\begin{equation}
\label{eq:marge} \bigl(s'_i/r_i -1\bigr)
\vert y_i\vert \to\infty,\qquad \bigl(1-s'_i/s_i
\bigr)\vert y_i\vert \to \infty.
\end{equation}
Let $n_i = (s'_i/r_i) \vert y_i\vert$. One has $y_i\in B(e,n_i)$ by
construction. The condition \eqref{eq:marge} ensures that, for any
$C$, for large enough $i$, a point $y$ with $\vert y\vert\leq
\vert y_i\vert+C$ belongs to $B(e, n_i)$. Taking a further subsequence of
$y_i$ if necessary, we can also assume that growth
conditions~\eqref{eq:main_cond0} and~\eqref{eq:main_cond} are
satisfied by $n_i$.

To get the divergence of $G'(z,y_i)/G'(e,y_i)$ for some point $z$, we
will choose the measures $\mu_i$ so that the limits of this sequence
are different along even and odd values of $i$ (with a limit of the
order of $1$ along odd $i$, and a small limit along a subsequence of
even $i$). For $i$ even, we let $\mu_i =
(\delta_{y_i}+\delta_{y_i^{-1}})/2$. The choice of $\mu_i$ for odd
$i$ is postponed, let us first see the consequences of our choice for
even $i$. The statements we will give now are valid for any choice of
$\mu_i$ for odd $i$, with the only restriction that it has to be a
probability measure, supported in $B(e,n_i)$.

Let us describe the asymptotics of $G(e, zy_i)$ for any fixed $z$.

\begin{lem}
\label{lem:existe_Phi}
There exists a function $\Phi\dvtx\Gamma\to(0,+\infty)$, tending
to $0$
at infinity, such that for every $z$ there exist infinitely many even
indices $i$ for which
\[
G(e,z y_i) \leq\Phi(z) e^{-r_i n_i}.
\]
\end{lem}

Let us stress that the function $\Phi$ does not depend on the choice
of $\mu_i$ for odd $i$.

\begin{pf*}{Proof of Lemma \ref{lem:existe_Phi}}
The idea is that, to go from $e$ to $zy_i$, the random walk will most
likely make one big jump of size $n_i$ (corresponding to the measure
$\mu_i$), with weight $e^{-r_i n_i}/2$, and several small jumps. If
$z$ is large enough, a large number of small jumps will be needed,
giving a small contribution $\Phi(z)$. The other cases (no big jump,
or too many big jumps) will have a very small contribution. In this
proof, $i$ will implicitly be restricted to even values.

For the rigorous computation, we start from the
bound~\eqref{eq:borne_basique} and cut the sum into several pieces.
We should specify in which piece a template $t=(a_0,i_1,\ldots,
a_\ell)$ goes.
\begin{itemize}
\item We put in $J_1$ the templates with $\max t
>i$.
\item We put in $J_2$ the templates where at least two jumps
$i_k$ are equal to $i$.
\item We put in $J_3$ the templates with $\max t < i$ for which a
trajectory can go from $e$ to $zy_i$.
\item Finally, we put in $J_4$ the remaining templates, that is,
those with a single jump of size $n_i$ and other shorter
jumps, for which a trajectory can go from $e$ to $zy_i$.
\end{itemize}
Denote by $\Sigma_p$ the sum corresponding to templates in $J_p$. We
will show that, for $p\leq3$, one has $\Sigma_p = o(e^{-r_i n_i})$
when $i$ tends to infinity, and that for infinitely many indices $i$
one has $\Sigma_4 \leq\Psi(z) e^{-r_i n_i}$ for some function $\Psi$
tending to $0$ at infinity. The result follows with $\Phi=2\Psi$.

Inequality~\eqref{eq:somme_large} implies that $\Sigma_1 \leq C
e^{-r_{i+1}n_{i+1}}$. As $r_{i+1}n_{i+1}
> r_i n_i +i+1$ by~\eqref{eq:main_cond0}, this is negligible compared
to $e^{-r_i n_i}$, as desired.

A template $t\in\Sigma_2$ can be decomposed as $t= t_1 \cdot(i)
\cdot t_2 \cdot(i) \cdot t_3$, for some templates $t_1$, $t_2$ and
$t_3$. Since the sum of the weights of all templates is bounded, we
obtain
\[
\Sigma_2 \leq C e^{-r_i n_i} C e^{-r_i n_i} C.
\]
This is again negligible with respect to $e^{-r_i n_i}$.

A template $t$ in $J_3$ satisfies $\vert t\vert \geq\vert zy_i\vert$ and
$\max t < i$. Hence,~\eqref{eq:somme_petite} gives the bound
$\Sigma_3 \leq C e^{-s_i \vert zy_i\vert}$. We have
\begin{eqnarray*}
s_i \vert zy_i\vert - r_i
n_i & \geq& s_i\bigl(\vert y_i\vert - \vert z
\vert\bigr) - r_i n_i = s_i\bigl(\vert
y_i\vert - \vert z\vert\bigr) -s'_i \vert
y_i\vert
\\
& =& s_i \biggl( \biggl(1-\frac{s'_i}{s_i} \biggr)\vert
y_i\vert - \vert z\vert \biggr).
\end{eqnarray*}
As $(1- s'_i/s_i) \vert y_i\vert \to\infty$ by~\eqref{eq:marge}, this
tends to infinity. Hence,
%
\begin{equation}
\label{eq:sizxi_negl} e^{-s_i \vert zy_i\vert} = o\bigl(e^{-r_i n_i}\bigr).
\end{equation}
This shows that $\Sigma_3$ is negligible with respect to $e^{-r_i
n_i}$.

It remains to estimate $\Sigma_4$. A template $t\in J_4$ can be
decomposed uniquely as $t=t_1 \cdot(i) \cdot t_2$, for some
templates $t_1$ and $t_2$ with maximum $<i$. If this template
contributes to $G(e, zy_i)$, then $zy_i$ can be written as $u
y_i^{\pm1} v$ with $\vert u\vert\leq\vert t_1\vert$ and $\vert
v\vert\leq
\vert t_2\vert$. Denote by $\phi_i(z)$ the minimum of the quantities
$\vert u\vert+\vert v\vert$ over all decompositions $zy_i = uy_i^{\pm
1} v$,
we get $\vert t_1\vert+\vert t_2\vert \geq\phi_i(z)$. In particular,
$\vert t_1\vert\geq\phi_i(z)/2$ or $\vert t_2\vert \geq\phi
_i(z)/2$. It
follows that
\[
\Sigma_4 \leq2 \biggl(\sum_{\max t_1 < i, \vert t_1\vert \geq
\phi
_i(z)/2}
\pi(t_1) \biggr) e^{-r_i n_i} \biggl( \sum
_{t_2} \pi(t_2) \biggr).
\]
The first sum is bounded by $Ce^{-s_i \phi_i(z)/2} \leq C e^{-r
\phi_i(z)/2}$ by~\eqref{eq:somme_petite}, and the last sum is
uniformly bounded. Hence,
\[
\Sigma_4 \leq C e^{-r\phi_i(z)/2} e^{-r_i n_i}.
\]

To conclude, we have to show that $\phi_i(z)$ is large for infinitely
many values of~$i$, if $z$ is far away from $e$. Let $A>0$, let us
denote by $B_i$ the set of $z$ that can be written as $u y_i^{\pm1}
v y_i^{-1}$ for some $u$ and $v$ with $\vert u\vert+\vert v\vert \leq
A$. The
set $B_i$ is finite, with cardinality at most $f(A)=2 (\Card
B(e,A))^2$. If $z\notin B_i$, it satisfies $\phi_i(z) > A$ by
definition. The points with $\limsup\phi_i(z) \leq A$ belong to
$\bigcup_n \bigcap_{i>n}B_i$. This is an increasing union of sets of
cardinality at most $f(A)$, hence it has cardinality at most $f(A)$.
This shows that, apart from finitely many exceptions, $\limsup
\phi_i(z)>A$; hence, $\Sigma_4 \leq C e^{-r A/2} e^{-r_i n_i}$ for
infinitely many $i$'s.
\end{pf*}

Let us fix a point $z$ away from the origin, so that $\Phi(z)$ is
suitably small (how small will be seen later in the proof). We now
define the measures $\mu_i$ for odd $i$. If $i$ is large enough,
$zy_i \in B(e, n_i)$ thanks to~\eqref{eq:marge}. For those $i$'s, let
\[
\mu_i = \tfrac{1}{4}( \delta_{y_i} +
\delta_{zy_i} + \delta _{y_i^{-1}} + \delta_{(z y_i)^{-1}}).
\]
The choice of $\mu_i$ for smaller $i$ is not relevant (take, e.g., $\mu_i=\delta_e$).

If $i$ is large and odd, Lemma~\ref{lem:ineg_i} gives $G(e, zy_i)
\leq Ce^{-r_i n_i} + C e^{-s_i \vert zy_i\vert}$.
By~\eqref{eq:sizxi_negl}, the second term is negligible with respect
to the first one. Hence, $G(e, zy_i) \leq C e^{-r_i n_i}$. In the
same way $G(e, y_i) \leq C e^{-r_i n_i}$.

The Green function $G'=G_{\mu'}$ is bounded by $G=G_\mu$. For $i$
large and odd, we obtain $G'(e, zy_i) \leq C e^{-r_i n_i}$ and
$G'(e,y_i) \leq C e^{-r_i n_i}$. As it is possible to jump directly
from $e$ to $zy_i$ or $y_i$ with weight $\mu(\Gamma)^{-1}e^{-r_i
n_i}/4$, corresponding lower bounds hold. In particular, there exists
a constant $C_0$ such that, for $i$ large and odd,
\[
\frac{G'(e, zy_i)}{G'(e,y_i)} \in\bigl[C_0^{-1}, C_0
\bigr].
\]

For infinitely many (even) values of $i$, we have $G'(e, zy_i) \leq
\Phi(z) e^{-r_i n_i}$ by Lemma~\ref{lem:existe_Phi}. Moreover, $G'(e,
y_i) \geq C^{-1} e^{-r_i n_i}$ [since one can jump directly from $e$
to $y_i$ with weight $\mu(\Gamma)^{-1} e^{-r_i n_i}/2$]. Hence, for
those values of $i$, there exists a constant $C_1$ such that
\[
\frac{G'(e, zy_i)}{G'(e,y_i)} \leq C_1 \Phi(z).
\]

We can finally specify the choice of $z$: as $\Phi$ tends to $0$ at
infinity, we may choose $z$ such that $C_1\Phi(z) < C_0^{-1}$. The
previous estimates imply that
\[
\liminf_i \frac{G'(e, zy_i)}{G'(e,y_i)} \leq C_1 \Phi(z)
< C_0^{-1} \leq\limsup\frac{G'(e, zy_i)}{G'(e,y_i)}.
\]
In particular, the sequence $G'(e, zy_i)/G'(e, y_i)$ does not
converge when $i$ tends to infinity. Equivalently,
$G'(z^{-1},y_i)/G'(e,y_i)$ does not converge. This completes the
proof of Theorem~\ref{thmm:pathological}. 

\section{Positive results in hyperbolic groups}
\label{sec:positive}

\subsection{Preliminaries}

A hyperbolic group is a finitely generated group in which geodesic
triangles are $\delta$-thin for some $\delta$, that is, each side of the
triangle is included in the $\delta$-neighborhood of the union of the
other sides. This notion is independent of the choice of the
generating system (albeit the constant $\delta$ does change with the
generating system). See, for instance,~\cite{ghys_hyperbolique}. This
essentially means that finite configurations of points in the group
resemble finite configurations of points in a tree---this intuition
is made precise by the following classical theorem.

\begin{thmm}
\label{thmm:tree_approx}
For any $n\in\N$ and $\delta>0$, there exists a constant
$C=C(n,\delta)$ with the following property. Consider a subset $A$ of
a $\delta$-hyperbolic group, of cardinality at most $n$. There exists
a map $\Phi$ from $A$ to a metric tree such that, for any $x,y\in A$,
\[
d(x,y)-C \leq d\bigl(\Phi(x), \Phi(y)\bigr) \leq d(x,y).
\]
\end{thmm}

Another intuition is that $\delta$-hyperbolic spaces resemble the
usual hyperbolic space $\Hbb^m$. Again, this is made precise by the
following theorem~\cite{bonk_schramm}. We will write $d_{\Hbb}$ for
the hyperbolic distance in $\Hbb^m$, and $\vert x\vert_{\Hbb} =
d_{\Hbb}(x, O)$ where $O$ is a fixed reference point in $\Hbb^m$.

\begin{thmm}
\label{thmm:bonk_schramm}
Consider a hyperbolic group $\Gamma$. If $m$ is large enough, there
exist a mapping $\Psi\dvtx \Gamma\to\Hbb^m$ and $\lambda>0$,
$C>0$ such
that, for all $x,y\in\Gamma$,
\[
\bigl\vert\lambda d_{\Hbb}\bigl(\Psi(x), \Psi(y)\bigr) - d(x,y)\bigr\vert\leq C.
\]
\end{thmm}

Ancona's original strategy~\cite{ancona2} to prove Ancona
inequalities~\eqref{eq:ancona} for finitely supported measures, based
on a subtle induction, is apparently difficult to extend to measures
with infinite support. We will rather rely on the strategy
of~\cite{gouezel_lalley}, and in particular on the following lemma
(see the proofs of Theorems~4.1 and~4.3 in~\cite{gouezel_lalley}). We
recall that the relative Green function $G(x,y;\Omega)$ has been
defined in Section~\ref{sec:green_function}.

\begin{definition}
Let $\mu$ be an admissible measure with finite Green function on a
hyperbolic group. It satisfies \emph{pre-Ancona inequalities} if, for
all $K>0$, there exists $n_0$ such that, for all $n\geq n_0$, for all
points $x,y,z$ on a geodesic segment (in this order) with $d(x,y)\in
[n,100n]$ and $d(y,z)\in[n,100 n]$, one has $G(x,z; B(y,n)^c) \leq
K^{-n}$.
\end{definition}

\begin{lem}
\label{lem:GL}
Let $\mu$ be an admissible measure on a hyperbolic group. Assume that
$\mu$ satisfies pre-Ancona inequalities. Then it satisfies Ancona
inequalities~\eqref{eq:ancona}.
\end{lem}

This lemma justifies the name ``pre-Ancona inequalities.'' It is
proved in~\cite{gouezel_lalley} as follows. Assume that $x,y,z$ are
given along a geodesic, and one wants to prove that $G(x,z)\leq
CG(x,y)G(y,z)$. One constructs a string of beads along a geodesic
segment $[x,z]$, the size of a bead being proportional to its
distance to $y$. Then, using pre-Ancona inequalities, one shows
inductively that the weight of trajectories avoiding any bead is
comparatively small. It follows that most weight comes from
trajectories passing in a bead within distance $O(1)$ of $y$, as
desired.

To prove Ancona inequalities, our strategy will always be to show
that pre-Ancona inequalities are satisfied.

\subsection{Ancona inequalities for measures satisfying
\texorpdfstring{$\mathrm{Anc}_*$}{Anc*}}

In this paragraph, we prove Theorem~\ref{thmm:old_ancona}. Consider an
admissible measure $\mu$ on a hyperbolic group, with superexponential
tails and satisfying $\mathrm{Anc}_*$, we will show that it satisfies
pre-Ancona
inequalities. We have to show that, for any points $x,y,z$ on a
geodesic in this order with $n\leq d(x,y), d(y,z) \leq100n$, the
Green function $G(x,z; B(y,n)^c)$ decays superexponentially fast in
terms of $n$.

We express things in terms of operators. Let $P=P_\mu$ be the
operator associated to $\mu$. We decompose $P$ as $A_n+B_n$ where
$A_n$ corresponds to jumps of size at most $n/2$, and $B_n$ to the
bigger jumps. On $\ell^2$, they satisfy $\Vert A_n\Vert \leq\Vert
P\Vert
\leq\mu(\Gamma)$ (which is finite since $\mu$ has well-defined
tails), and $\Vert B_n\Vert$ decays superexponentially fast in terms of
$n$ by assumption.

Let us fix a constant $C_0$. The Green function $G(x,z)$ is the sum
of the weights $\pi(\tau)$ of all paths $\tau$ from $x$ to $z$. The
contribution of paths with length at most $C_0 n$ is
\begin{eqnarray*}
\sum_{k=0}^{C_0 n} P^k
\delta_z(x) & = &\sum_{k=0}^{C_0 n}
(A_n+B_n)^k \delta_z(x) \leq
\sum_{k=0}^{C_0 n}\bigl \Vert(A_n+B_n)^k
\bigr\Vert
\\
& \leq&\sum_{k=0}^{C_0 n} \sum
_{\ell=0}^k \pmatrix{k
\cr
\ell} \Vert A_n
\Vert^\ell\Vert B_n\Vert^{k-\ell}.
\end{eqnarray*}
By $\mathrm{Anc}_*$, there exists $r>1$ such that the measure $r\mu$
has a finite
Green function. The contribution to $G(x,z)$ of paths longer than
$C_0 n$ is
\[
\sum_{k>C_0 n} p^k(x,z) \leq
r^{-C_0 n} \sum_{k>C_0 n} r^k
p^k(x,z) \leq r^{-C_0 n} G_{r \mu}(x,z).
\]
The quantity $G_{r\mu}(x,z)$ grows at most exponentially in terms of
$n$, thanks to Harnack inequality~\eqref{eq_harnack} and since
$d(x,z) \leq200 n$. Hence, we obtain from some constant $D_0$
independent of $C_0$
%
\begin{equation}
\label{eq:G} G(x,z) \leq\sum_{k=0}^{C_0 n}
\sum_{\ell=0}^k \pmatrix{k
\cr
\ell} \Vert
A_n\Vert^\ell\Vert B_n\Vert^{k-\ell} +
r^{-C_0 n} D_0^n.
\end{equation}

Let us now estimate $G(x,z; B(y,n)^c)$. Consider a trajectory from
$x$ to $z$ outside of $B(y,n)$ with jumps bounded by $n/2$. Putting
geodesics between the successive points of the trajectory, one
obtains a path from $x$ to $z$ avoiding $B(y,n/2)$. This path is
exponentially long (since this is the case in hyperbolic space, to
which the group can be compared thanks to
Theorem~\ref{thmm:bonk_schramm}). Hence, the number of jumps is at
least $C e^{\alpha n} /n\geq C e^{\beta n}$. It follows that, among
trajectories of length at most $C_0 n$, it is necessary to have a
jump larger than $n/2$ if $n$ is large enough. This shows that,
in~\eqref{eq:G}, the terms with $k=\ell$ (i.e., coming from $A_n^k$)
do not contribute to $G(x,z; B(y,n)^c)$. This equation gives
\[
G\bigl(x,z; B(y,n)^c\bigr) \leq\sum_{k=0}^{C_0 n}
\sum_{\ell=0}^{k-1} \pmatrix{k
\cr
\ell} \Vert
A_n\Vert^\ell\Vert B_n\Vert^{k-\ell} +
r^{-C_0 n} D_0^n.
\]
As $k-\ell\geq1$, we can bound $\Vert B_n\Vert^{k-\ell}$ with
$\Vert B_n\Vert$, yielding
\begin{eqnarray*}
G\bigl(x,z; B(y,n)^c\bigr) &\leq&\Vert B_n\Vert \sum
_{k=0}^{C_0 n} \sum
_{\ell=0}^{k-1} \pmatrix {k
\cr
\ell} \Vert A_n
\Vert^\ell + r^{-C_0 n} D_0^n
\\
& \leq&\Vert B_n\Vert \sum_{k=0}^{C_0 n}\bigl(
\Vert A_n\Vert+1\bigr)^k + r^{-C_0
n}
D_0^n
\\
& \leq&\Vert B_n\Vert \sum_{k=0}^{C_0 n}\bigl(
\Vert P\Vert+1\bigr)^k + r^{-C_0
n} D_0^n
\\
& \leq&\Vert B_n\Vert D_1^{C_0 n} +
r^{-C_0 n} D_0^n,
\end{eqnarray*}
for some constant $D_1$ independent of $C_0$.

Let us complete the proof. Fix $K>1$, we want to show that
$G(x,z;\break B(y,n)^c) \leq2K^{-n}$ if $n$ is large enough. First, we
choose $C_0$ with $r^{-C_0} D_0 < K^{-1}$, so that the second term in
the previous equation is bounded by $K^{-n}$. Then, as $\Vert B_n\Vert$
decays superexponentially, we have $\Vert B_n\Vert D_1^{C_0 n} \leq
K^{-n}$ if $n$ is large enough. 

\begin{rmk}
If the measure $\mu$ has finite support, the proof simplifies
drastically since there is no trajectory from $x$ to $z$ with length
at most $C_0 n$ avoiding $B(y,n)$. Hence, one gets a very simple
proof of Ancona's original results~\cite{ancona2} (most of the
complexity is in fact hidden in Lemma~\ref{lem:GL}).
\end{rmk}

\subsection{Ancona inequalities in the free group}
\label{subsec:free}

In this paragraph, we prove the second item of
Theorem~\ref{thmm:new_ancona}: in a free group, an admissible measure
$\mu$ with superexponential tails and finite Green function satisfies
Ancona inequalities. Since Ancona inequalities for finitely supported
measures are trivial in the free group, the only difficulty comes
from long jumps. The trick we will devise to handle those long jumps
(replacing a trajectory involving a long jump by a longer trajectory
with short jumps) will be used several times in the rest of the
paper.

By Lemma~\ref{lem:GL}, it suffices to show that $\mu$ satisfies
pre-Ancona inequalities. Consider three points $x,y,z$ on a geodesic
in this order with $n\leq d(x,y), d(y,z)\leq100n$, we want to show
that $G(x,z; B(y,n)^c)$ is superexponentially small. We may assume
without loss of generality that $y=e$. We will first give the proof
assuming for simplicity that $\mu$ gives positive mass to every
generator of the group.

Denote by $Z_0,\ldots, Z_N$ the finitely many connected components of
$\Gamma- B(e, n/2)$, with $x\in Z_0$ and $z\in Z_N$. Let also $A_i =
Z_i\cap(\Gamma- B(e,n))$.

Consider a trajectory $\tau=(x_0=x, x_1,\ldots, x_{k-1}, x_k=z)$ of
the random walk from $x$ to $z$, avoiding $B(e,n)$. It cannot stay
forever in $A_0$, let us say that the first jump outside of $A_0$ is
from $x_i$ to $x_{i+1}$. We associate to $\tau$ a modified trajectory
$m(\tau)$ (again from $x$ to $z$) as follows. Let $a$ and $b$ be
different elements in the support of $\mu$. Let $\tau_i$ be a
geodesic from $x_i$ to $e$, with length $n_i = \vert x_i\vert$, and let
$\tau_{i+1}$ be a geodesic from $e$ to $x_{i+1}$, with length
$n_{i+1}=\vert x_{i+1}\vert$. We let
%
\begin{eqnarray}
\label{eq:def_m_tau}&& m(\tau) = \bigl(x_0,\ldots, x_{i-1}, (
\tau_i), a,a^{-1},\ldots, a, a^{-1},
\nonumber
\\[-8pt]
\\[-8pt]
\nonumber
&&\hspace*{41pt}b,b^{-1},\ldots, b,b^{-1}, (\tau_{i+1}),
x_{i+2},\ldots, x_k=z\bigr),
\end{eqnarray}
where we put $n_i$ copies of $a, a^{-1}$ and $n_{i+1}$ copies of $b,
b^{-1}$. The interest of this insertion is that the map $\tau\to
m(\tau)$ is one-to-one: if one knows $m(\tau)$, then the number of
$a,a^{-1}$ following the first return to $e$ gives $n_i$. In the same
way, one can determine $n_{i+1}$. Removing the pieces of length
$n_i-1$ before the first return to $e$, and $n_{i+1}-1$ after the
last return to $e$, one recovers the initial trajectory $\tau$.

To get $m(\tau)$,\vspace*{1pt} we removed a big jump of $\tau$, and we added
$3(n_i+n_{i+1})$ jumps of length $1$ (with weight uniformly bounded
from below, by a constant $C_0^{-1}$). We obtain
\[
\pi\bigl(m(\tau)\bigr) \geq\pi(\tau) C_0^{-3(n_i+n_{i+1})} /
\pi(x_i, x_{i+1}).
\]
For any constant $K$, there exists $C_K$ such that $\pi(e,u)=\mu(u)
\leq C_K K^{-\vert u\vert}$ since $\mu$ has superexponential tails.
Hence, we get
\[
\pi(\tau) \leq\pi\bigl(m(\tau)\bigr) C_0^{3(n_i+n_{i+1})}
C_K K^{- d(x_i, x_{i+1})}.
\]
Since $x_i$ and $x_{i+1}$ belong to different connected components of
$\Gamma- B(e,n/2)$, we have $d(x_i, x_{i+1}) \geq
\vert x_i\vert+\vert x_{i+1}\vert-n$. As $\vert x_i\vert\geq n$ and
$\vert x_{i+1}\vert\geq n$, this gives $d(x_i, x_{i+1}) \geq
(\vert x_i\vert+\vert x_{i+1}\vert)/2=(n_i+n_{i+1})/2$. We get
\[
\pi(\tau) \leq\pi\bigl(m(\tau)\bigr) C_0^{3(n_i+n_{i+1})}
C_K K^{- (n_i +
n_{i+1})/2}.
\]
If $K$ is large enough so that $C_0^3 K^{-1/4} \leq1$, we obtain
\[
\pi(\tau) \leq\pi\bigl(m(\tau)\bigr) C_K K^{- (n_i + n_{i+1})/4} \leq\pi
\bigl(m(\tau)\bigr)C_K K^{-n/2}.
\]
The map $\tau\mapsto m(\tau)$ is one-to-one. Summing over all
trajectories from $x$ to $z$ outside of $B(e,n)$, we obtain
\[
G\bigl(x,z; B(e,n)^c\bigr) \leq C_K K^{-n/2}
G(x,z).
\]
Since $d(x,z)\leq200 n$, we have $G(x,z) \leq C^n$ by Harnack
inequalities~\eqref{eq_harnack}. As $K$ can be arbitrarily large,
this shows that $G(x,z; B(e,n)^c)$ is smaller than any exponential,
as desired. This completes the proof of pre-Ancona inequalities when
$\mu$ gives positive mass to all generators.

In the general case, one has to tweak the definition of the modified
trajectory $m(\tau)$ to ensure that $m(\tau)$ has positive weight,
while retaining the injectivity of the map $\tau\mapsto m(\tau)$.
One can, for instance, proceed as follows. To each generator $s$, let
us associate a path $\sigma_s$ from $e$ to $s$ with $\pi(\sigma_s)>0$---such
a path exists since $\mu$ is admissible. Then, in the
definition of $m(\tau)$, one replaces the geodesic $\tau_i=s_1 \cdots
s_{n_i}$ with the concatenation $\tilde\tau_i$ of the paths
$\sigma_{s_1}\cdots\sigma_{s_{n_i}}$. In the same way, one replaces
$\tau_{i+1}$ with the corresponding path $\tilde\tau_{i+1}$. Note
that $\pi(\tilde\tau_i) \geq C_1^{-n_i}$ and $\pi(\tilde\tau_{i+1})
\geq C_1^{-n_{i+1}}$ for some constant $C_1$, since the lengths of
$\tilde\tau_i$ and $\tilde\tau_{i+1}$ are bounded, respectively, by
$Cn_i$ and $Cn_{i+1}$.

A problem that may appear with this construction is that the first
return to $e$ in $m(\tau)$ might happen before the end of $\tilde
\tau_i$, so that the reconstitution of $\tau$ from $m(\tau)$ is
problematic. To avoid this problem, one may add a loop $\gamma$ from
$e$ to itself, with $\pi(\gamma)>0$, that does not appear when one
concatenates paths $\sigma_s$ along a geodesic segment. In the end,
one chooses for $m(\tau)$ the trajectory
%
\begin{eqnarray}
\label{eq:def_m_tau_complicated}&& \bigl(x_0,\ldots, x_{i-1}, (\tilde
\tau_i), (\gamma), (\alpha),\ldots, (\alpha), (\beta),\ldots, (
\beta),
\nonumber
\\[-8pt]
\\[-8pt]
\nonumber
&&\hspace*{97pt}(\gamma), (\tilde\tau_{i+1}), x_{i+2},\ldots,
x_k=z\bigr),
\end{eqnarray}
where $\alpha$ and $\beta$ are two fixed distinct loops from $e$ to
$e$ with positive weight, and one puts $\vert\tilde\tau_i\vert$ terms
$\alpha$ and $\vert\tilde\tau_{i+1}\vert$ terms $\beta$. By
construction, $\tau\mapsto m(\tau)$ is one-to-one and $\pi(m(\tau))
\geq\pi(\tau) C_2^{n_i+n_{i+1}}/\pi(x_i, x_{i+1})$ for some constant
$C_2$. The rest of the argument goes through. 

\subsection{Ancona inequalities for symmetric measures}
\label{subsec:preAncona_symmetric}

In this paragraph, we prove the first item of
Theorem~\ref{thmm:new_ancona}: in a hyperbolic group, a symmetric
admissible measure $\mu$ with superexponential tails and finite Green
function satisfies Ancona inequalities. By Lemma~\ref{lem:GL}, it
suffices to show that it satisfies pre-Ancona inequalities. Consider
three points $x,y,z$ on a geodesic in this order with $n\leq d(x,y),
d(y,z)\leq100 n$; we want to show that $G(x,z; B(y,n)^c)$ is
superexponentially small. We may assume without loss of generality
that $y=e$.

The proof follows the strategy in~\cite{gouezel_higherancona}, Theorem~2.3: we will construct several barriers so
that most trajectories from $x$ to $z$ will visit them. The
construction is made in $\Hbb^m$, using an approximate embedding
$\Psi$ of $\Gamma$ inside $\Hbb=\Hbb^m$ given by
Theorem~\ref{thmm:bonk_schramm}. We will think of $\Hbb^m$ using the
model of the unit ball in $\R^m$, hence its boundary is identified
with the unit sphere $S^{m-1}$. We denote by $O$ the center of the
unit ball in $\R^m$. Changing the generators of the group if
necessary, we may assume that $\mu$ gives positive mass to all of
them. We will need to choose at some point in the proof some very
small $\epsilon$, and we will denote by $C$ a generic constant that
does not depend on $\epsilon$.

We will use the following easy lemma of hyperbolic geometry.

\begin{lem}
\label{lem:angle_separe}
There exist $\alpha>0$ and $C>0$ with the following property: for any
points $a$ and $b$ in a ball $B_{\Hbb}(u, \vert u\vert_{\Hbb}/9)$ of
$\Hbb^m$, the angle between $[Oa]$ and $[Ob]$ is at most $C
e^{-\alpha\vert u\vert_{\Hbb}}$.
\end{lem}

The hyperbolic geodesic from $\Psi(x)$ to $\Psi(z)$ can be extended
biinfinitely. Composing $\Psi$ with a hyperbolic isometry, we can
assume that the center $O$ of the unit ball in $\R^m$ belongs to this
geodesic, and that $\Psi(e)$ is at a bounded distance of $O$. Let
$\xi$ denote the limit in negative time of this geodesic.

To an angle $\theta\in(0,\pi)$, we associate the union $Y(\theta)$
of all semiinfinite geodesics $[O\zeta)$ (with $\zeta\in S^{m-1}$)
making an angle $\theta$ with $[O\xi)$ (its boundary at infinity is
the set of points of $S^{m-1}$ at distance $\theta$ of $\xi$). This
is the boundary of a cone based at $O$. Let $Z(\theta)$ be the union
of all hyperbolic balls $B_{\Hbb}(u, \vert u\vert_{\Hbb}/10)$ for
$u\in
Y(\theta)$. This is a thickening of $Y(\theta)$, thicker and thicker
close to infinity. It cuts $\Hbb^m$ into two connected components.

\begin{lem}
\label{lem:taille_saut}
If $u$ and $v$ are two points in the two components of $\Hbb^m -
Z(\theta)$, one has
\[
d_{\Hbb}(u,v) \geq\bigl(\vert u\vert_{\Hbb} + \vert v
\vert_{\Hbb}\bigr)/11.
\]
\end{lem}

\begin{pf}
The hyperbolic geodesic from $u$ to $v$ intersects $Y(\theta)$ at a
single point $w$. It satisfies
$d_{\Hbb}(u,v)=d_{\Hbb}(u,w)+d_{\Hbb}(w,v)$. By assumption, $u\notin
B_{\Hbb}(w, \vert w\vert_{\Hbb}/10)$, hence $d_{\Hbb}(u,w) \geq
\vert w\vert_{\Hbb}/10$. Trivially, $d_{\Hbb}(u,w) \geq
\vert u\vert_{\Hbb}-\vert w\vert_{\Hbb}$. For any $t\in[0,1]$, we obtain
\[
d_{\Hbb}(u,w) \geq t \vert w\vert_{\Hbb}/10 + (1-t) \bigl(\vert u
\vert _{\Hbb
}-\vert w\vert_{\Hbb}\bigr).
\]
Let $t=10/11$, so that the terms involving $\vert w\vert_{\Hbb}$ cancel
each other. We are left with $d_{\Hbb}(u,w) \geq\vert u\vert_{\Hbb}/11$.
Since an analogous estimate is true for $v$, this completes the
proof.
\end{pf}

Let $A(\theta) = B(e,n)^c \cap\Psi^{-1}(Z(\theta))\subset\Gamma$ be
the set of points of $\Gamma$ outside of $B(e,n)$ whose image under
$\Psi$ belongs to $Z(\theta)$. The previous lemma shows that, if a
trajectory in $\Gamma$ jumps past $A(\theta)$, it has to make a big
jump.

Let $N = \lfloor e^{\epsilon n}\rfloor$. In $X=[0,\pi]$, let $X_i =
[(2i-1)/N, 2i/N]$ for $1\leq i\leq N$. For any $\theta_i\in X_i$ and
$\theta_{i+1}\in X_{i+1}$, the visual angle from $O$ between two
points in $Y(\theta_i)$ and $Y(\theta_{i+1})$ is at least
$e^{-\epsilon n}$. It follows from Lemma~\ref{lem:angle_separe} that,
if $\epsilon$ is small enough and if $n$ is large enough, the angle
between two points in $Z(\theta_i)$ and $Z(\theta_{i+1})$ is at least
$e^{-\epsilon n}/2$. This shows in particular that $A(\theta_i)$ and
$A(\theta_{i+1})$ are disjoint.

\begin{lem}
\label{lem:barrieres}
If $\epsilon$ is small enough, there exist angles $\theta_i\in X_i$
such that, for all $0\leq i\leq N$,
%
\begin{equation}
\label{eq:barrieres} \sum_{u\in A_i, v\in A_{i+1}} G(u,v)^2
\leq1/4,
\end{equation}
where $G$ is the Green function associated to $\mu$ and we denoted
$A_0=\{x\}$, $A_{N+1}=\{z\}$ and $A_i= A(\theta_i)$ for $1\leq i\leq
N$.
\end{lem}

This lemma shows that one can choose barriers so that the weight of
trajectories going from one barrier to the next is small. This will
guarantee that trajectories visiting all barriers have a
superexponentially small weight. It will remain to handle
trajectories jumping past barriers---we will use
Lemma~\ref{lem:taille_saut} to show that the jumps have to be large,
implying that these trajectories contribute again with a very small
weight thanks to the argument of Section~\ref{subsec:free}.

\begin{pf*}{Proof of Lemma~\ref{lem:barrieres}}
The proof is similar to that of Lemma~2.6
in \cite{gouezel_higherancona}; the difference is that we are
considering thicker barriers. For $a\in\Gamma$, let $X_i(a)$ be the
set of angles $\theta\in X_i$ such that $a\in A(\theta)$. If one
shows that
%
\begin{equation}
\label{eq:interm} \Leb\bigl(X_i(a)\bigr) \leq C e^{-\alpha\vert a\vert}
\end{equation}
for some $\alpha$ independent of $\epsilon$, the remaining part of
the argument of~\cite{gouezel_higherancona} will apply verbatim. We
sketch very quickly the rest of the argument
in~\cite{gouezel_higherancona} for the convenience of the reader.

Using hyperbolicity, one checks that a supermultiplicative function\break
$H$ with $\sum_{x\in\Gamma} H(e,x)<\infty$ has bounded sum on any
sphere $\Sbb^k$, that is,\break $\sum_{x\in\Sbb^k} H(e,x) \leq C$ uniformly
in $k$, where $C$ does not depend on $H$. This estimate applies to
$H_r(e,x) = G_{r\mu}(e,x)G_{r\mu}(x,e)$ for any $r<1$. Letting $r$
tend to $1$ and using the symmetry of $\mu$, we obtain $\sum_{x\in
\Sbb^k} G(e,x)^2 \leq C$. Hence, the function $G(e,x)$ is not in
$\ell^2(\Gamma)$, but close. In particular, if $A$ is a subset such
that $\Card(A\cap\Sbb^k)$ is exponentially smaller than $\Sbb^k$,
one expects that typically $\sum_{x\in A} G(e,x)^2$ will be finite
(and small if $A$ is thin enough). Of course, this might not be true
for all such subsets $A$, but it will be true for most subsets $A$ in
a suitable sense. The lemma is proved by showing that, if one chooses
$\theta_i$ randomly in $X_i$, then the estimate~\eqref{eq:barrieres}
holds with positive probability. This follows from the combination of
inequality~\eqref{eq:interm} with the estimate $\sum_{x\in
\Sbb^k} G(e,x)^2 \leq C$.

It remains to prove~\eqref{eq:interm}. Since distances in the group
and in hyperbolic space are equivalent, it is sufficient to show the
corresponding estimate in $\Hbb$, that is, for all $u\in\Hbb$,
\[
\Leb\bigl\{\theta\st u\in Z(\theta)\bigr\} \leq C e^{-\alpha\vert u\vert
_{\Hbb}}.
\]
For $u\in Z(\theta)$, there exists $v\in Y(\theta)$ such that
$d_{\Hbb}(u,v) \leq\vert v\vert_{\Hbb}/10$. Since $\vert v\vert
_{\Hbb}/10
\leq(d_{\Hbb}(u,v) + \vert u\vert_{\Hbb})/10$, we get
$d_{\Hbb}(u,v)\leq\vert u\vert_{\Hbb}/9$, that is, $v\in B_{\Hbb}(u,\break
\vert u\vert_{\Hbb}/9)$. Lemma~\ref{lem:angle_separe} shows that the
trace at infinity of this ball gives rise to an exponentially small
angle. This completes the proof.
\end{pf*}

Let us prove the pre-Ancona inequalities. The Green function
$G(x,z;\break B(e,n)^c)$ is the sum of the weights $\pi(\tau)$ of the
trajectories $\tau$ from $x$ to $z$ avoiding $B(e,n)$. We will say
that such a trajectory is walking if it visits in this order the
barriers $A_1,\ldots,A_N$ constructed in Lemma~\ref{lem:barrieres},
and jumping otherwise.

Decomposing walking trajectories according to their first visits to
the barriers, we get that their contribution to $G(x,z; B(e,n)^c)$ is
bounded by
\[
\sum_{a_1 \in A_1,\ldots, a_N \in A_N} G(x, a_1) G(a_1,
a_2) \cdots G(a_{N-1}, a_N)
G(a_N, z).
\]
Using the estimate~\eqref{eq:barrieres} on barriers and
Cauchy--Schwarz inequality, one shows that this is bounded by $2^{-N}
\leq2^{-e^{\epsilon n}+1}$ (see the beginning of the proof of
Lemma~2.6 in~\cite{gouezel_higherancona}). Hence, the contribution of
walking trajectories is smaller than any exponential, as desired.

Consider now a jumping trajectory $\tau=(x_0=x, x_1,\ldots, x_{k-1},
x_k=z)$, and assume that the first jump past a barrier happens at
index $i$, from $x_i$ to $x_{i+1}$. One associates to $\tau$ a
modified trajectory $m(\tau)$ as in Section~\ref{subsec:free} [see
equation~\eqref{eq:def_m_tau} there---as we assume that $\mu$ gives
positive weight to the generators, there is no need to use the more
complicated definition~\eqref{eq:def_m_tau_complicated}].
Lemma~\ref{lem:taille_saut} shows that there exists a constant $C$
such that $d(x_i, x_{i+1}) \geq C^{-1} (\vert x_i\vert+\vert
x_{i+1}\vert)$.
This is sufficient for all the computations of
Section~\ref{subsec:free}. It follows that the contribution of
jumping trajectories is smaller than any exponential, as desired.

\subsection{Ancona inequalities in Fuchsian groups}

In this paragraph, we prove the third item of
Theorem~\ref{thmm:new_ancona}: an admissible measure $\mu$ with
superexponential tails and finite Green function on a cocompact
lattice $\Gamma$ of ${\mathrm{PSL}(2,\R)}$ satisfies Ancona
inequalities. Since the
argument follows rather closely the previous subsection, we will only
sketch the argument. Note that $\Gamma$ is quasi-isometric with
$\Hbb^2$, giving an identification of the boundary $\partial\Gamma$
with the circle $S^1$. The planarity of $\Hbb^2$ will be essential.

Again, we want to prove pre-Ancona inequalities between points $x$,
$y$ and $z$ with $n\leq d(x,y),d(y,z) \leq100 n$, and we may assume
that $y=e$. As in the previous subsection, we will construct several
barriers between $x$ and $z$, and treat separately trajectories that
visit all the barriers (walking trajectories) and trajectories that
jump past a barrier (jumping trajectories).

The basic ingredient for the barriers is constructed
in~\cite{gouezel_higherancona}, Appendix A: it is shown there that,
for any finite family of disjoint subintervals
$I^{(1)},\ldots,I^{(N)}$ of $S^1$, one can find for $1\leq i\leq N$
paths $X^{(i)}_n$ in the Cayley graph of $\Gamma$ starting from $e$
such that:
\begin{itemize}
\item One has $d(X^{(i)}_k, X^{(i)}_{k+1}) \leq1$.
\item The path $X^{(i)}_k$ converges to a point in $I^{(i)}$ when
$k\to\infty$.
\item There exist $\alpha>0$ and $C>0$ such that
%
\begin{equation}\quad\hspace*{4pt}
\label{ineq:G_fuchsian} G\bigl(e, X^{(i)}_k\bigr) \leq C
e^{-\alpha k}\quad \mbox{and} \quad G\bigl(X^{(i)}_k,
X^{(j)}_\ell\bigr) \leq C e^{-\alpha(k+\ell)}\qquad
\mbox{for all }i
\neq j.
\end{equation}
\item For some $s>0$, one has $d(e,X^{(i)}_k) \sim s k$.
\end{itemize}
The constant $C$ in the third item depends on $N$, while the other
constants do not. The paths $X^{(i)}_k$ are constructed as typical
trajectories of \emph{another} (symmetric) random walk. The
inequalities for $G$ only rely on the
supermultiplicativity~\eqref{eq:super_mult} of the Green function of
$\mu$ (and a version of Kingman's ergodic theorem)---in particular,
the finiteness of the support of $\mu$ is not required.

Given such trajectories, one can replace each point $X^{(i)}_k$ by a
ball $B(X^{(i)}_k, C)$ of some fixed radius $C$. This yields barriers
that random walks with finite range cannot avoid, as
in~\cite{gouezel_higherancona}. The
inequalities in~\eqref{ineq:G_fuchsian} guarantee that such barriers
satisfy an inequality similar to~\eqref{eq:barrieres}. However, such
a thickening does not imply that a jump past the barrier has to be
long. Let us define a thicker barrier by $Z_i = \bigcup_k
B(X^{(i)}_k, ck)$, where $c\leq1$ is a suitably small constant, and
let $A_i = Z_i \cap(\Gamma- B(e,n))$.

As in Lemma~\ref{lem:taille_saut}, one shows that jumps above such
barriers have to be long. It follows that jumping trajectories will
give a contribution to $G(x,z; B(e,n)^c)$ that is smaller than any
exponential, as in Section~\ref{subsec:free}.

To control the contribution of walking trajectories, it only remains
to prove that an inequality similar to~\eqref{ineq:G_fuchsian} holds:
if $n$ is large enough,
%
\begin{equation}
\label{eq:control_fuchsian} \sum_{u\in A_i, v\in A_j} G(u,v)^2
\leq1/4.
\end{equation}
To prove this estimate, consider two points $u$ and $v$ in $A_i$ and
$A_j$. They belong to balls $B(X^{(i)}_k, ck)$ and $B(X^{(j)}_\ell,
c\ell)$. Note first that
\[
n\leq\vert u\vert \leq\bigl\vert X^{(i)}_k\bigr\vert + ck \leq(1+c)
k.
\]
In particular, $k \geq n/2$. In the same way, $\ell\geq n/2$. Thanks
to Harnack inequalities~\eqref{eq_harnack}, we have
\[
G(u,v) \leq C_0^{d(u, X^{(i)}_k)}C_0^{d(X^{(j)}_\ell, v)} G
\bigl(X^{(i)}_k, X^{(j)}_\ell\bigr) \leq
C_0^{ck+c\ell} C e^{-\alpha(k+\ell)}.
\]
If $c$ is small enough, this is bounded by $C e^{-\alpha(k+\ell)/2}$.
Hence, we get
\[
\sum_{u\in A_i, v\in A_j} G(u,v)^2 \leq C \sum
_{k,\ell\geq n/2} \Card B\bigl(X^{(i)}_k,
ck\bigr) \Card B\bigl(X^{(j)}_\ell, c\ell\bigr) C
e^{-\alpha(k+\ell)}.
\]
If $c$ is small enough, $\Card B(X^{(i)}_k, ck) = \Card B(e, ck)$
grows at most like $e^{\alpha k/2}$. The
estimate~\eqref{eq:control_fuchsian} follows for large $n$. 

\subsection{Strong Ancona inequalities}

The proof of Theorem~\ref{thmm:local} on the asymptotics of transition
probabilities involves a reinforcement of Ancona inequalities, called
strong Ancona inequalities and defined as follows.

\begin{definition}
\label{def_strong_ancona}
An admissible measure $\mu$ with finite Green measure on a hyperbolic
group satisfies strong Ancona inequalities if it satisfies Ancona
inequalities and, additionally, there exist constants $C>0$ and
$\rho>0$ such that, for all points $x,x',y,y'$ whose configuration is
approximated by a tree as follows:\vspace*{6pt}

\includegraphics{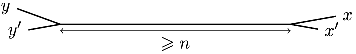}
\vspace*{6pt}

\begin{center}
 \begin{tikzpicture}
 \draw[thick]
 (0,0) -- +(160:1cm) node[left] {$y$}
 (0,0) -- +(190:0.7cm) node[left] {$y'$}
 (0,0) -- (5,0) coordinate (y0)
 (y0) -- +(10:1cm) node[right] {$x$}
 (y0) -- +(-10:0.6cm) node[right] {$x'$};
 \draw[<->] (0,-0.14) -- node[below] {$\geq n$} (5, -0.14);
 \end{tikzpicture}\vspace*{6pt}
\end{center}
\noindent one has
%
\begin{equation}
\label{eq:strong_ancona}
\biggl\vert \frac{G(x,y)/G(x',y)}{G(x,y')/G(x',y')} -1\biggr\vert \leq C e^{-\rho n}.
\end{equation}
\end{definition}

Usual Ancona inequalities ensure that
$(G(x,y)/G(x',y))/(G(x,y')/\break G(x',y'))$ [the quantity on the left-hand
side of~\eqref{eq:strong_ancona}] is bounded from above and from
below. Strong Ancona inequalities strengthen this by saying that it
is exponentially close to $1$, in terms of the distance between
$\{x,x'\}$ and $\{y,y'\}$.

In this paragraph, we will prove the following theorem.

\begin{thmm}
\label{thmm:strong_ancona}
In a hyperbolic group $\Gamma$, consider an admissible measure $\mu$
with finite Green function and superexponential tails. Assume that
$\mu$ satisfies pre-Ancona inequalities. Then it satisfies strong
Ancona inequalities.
\end{thmm}

Quantitative inequalities such as strong Ancona inequalities are
instrumental to get asymptotics of transition probabilities. Indeed,
the following holds. Consider an admissible symmetric probability
measure $\mu$ on a hyperbolic group, let $R$ denote the inverse of
the spectral radius of the corresponding random walk, and assume that
the measures $r\mu$ (for $1\leq r\leq R$) satisfy strong Ancona
inequalities, uniformly in $r$ (i.e., with the same $C$ and the same
$\rho$). If the random walk generated by $\mu$ is aperiodic, it
follows that $p^n(x,y) \sim C(x,y) R^{-n}n^{-3/2}$ for all $x,y\in
\Gamma$. If $\mu$ is periodic, this is true for even $n$ (resp., odd
$n$) if the distance from $x$ to $y$ is even (resp., odd). This
statement follows from~\cite{gouezel_lalley}, Theorem~9.1
and~\cite{gouezel_higherancona}, Theorem~3.1.

\begin{pf*}{Proof of Theorem~\ref{thmm:local}}
Consider an admissible symmetric probability measure $\mu$ with
superexponential tails in a hyperbolic group $\Gamma$. Let $R$ denote
the inverse of its spectral radius.

It follows from the discussion in the previous paragraph that, to
prove Theorem~\ref{thmm:local}, it suffices to prove strong Ancona
inequalities for the measures $r\mu$, uniformly in $1\leq r\leq R$.
Pre-Ancona inequalities have been proved in
Section~\ref{subsec:preAncona_symmetric} for each of those
measures, hence they also satisfy strong Ancona inequalities by
Theorem~\ref{thmm:strong_ancona}. The only remaining problem is the
uniformity of those inequalities for $1\leq r \leq R$. One checks in
the proof of Theorem~\ref{thmm:strong_ancona} that the constants $C$
and $\rho$ one obtains only depend on the constants in the pre-Ancona
inequalities and in the Harnack inequalities. The pre-Ancona
inequalities for $R\mu$ imply the same inequalities for $r\mu$ for
any $r$, since $r\mu\leq R\mu$. Hence, the pre-Ancona inequalities
are uniform. Moreover, it is clear that the Harnack inequality are
also uniform in~$r$.
\end{pf*}

The rest of this subsection is devoted to the proof of
Theorem~\ref{thmm:strong_ancona}. The argument dates back to Anderson
and Schoen~\cite{anderson_schoen}. For finitely supported measures,
the methods of~\cite{anderson_schoen} were adapted to the free group
by Ledrappier~\cite{ledrappier_freegroup}, and then to any hyperbolic
group by Izumi, Neshveyev and Okayasu~\cite{izumi_hyperbolic}. The
idea is to define a sequence of shrinking domains on which two given
positive harmonic functions (with a common normalization) have to be
closer and closer, by an inductive argument: one shows that two
positive harmonic functions defined on one of those domains have a
common significant part on a smaller domain. One can then subtract
this common part to both functions in the smaller domain, and repeat
the argument. In particular, one always works with positive harmonic
functions, but defined on smaller and smaller domains.

While we will essentially follow the same strategy, the difficulty in
the case of infinitely supported measures is that harmonicity becomes
a global property, involving the whole group: it will not be possible
to work with functions defined only on subdomains, we will need to
keep track of the behavior of functions in the whole group. We will
retain positivity in the smaller domains, but we will also need
quantitative controls everywhere in the group.

The proof will involve not only global Ancona inequalities, but also
Ancona inequalities for Green functions restricted to some classes of
domains (as defined in Section~\ref{sec:green_function}).

\begin{definition}
Let $H_0$ be a constant. Let $[x,z]$ be a geodesic in $\Gamma$, and
let $y\in[x,z]$. We say that a subset $\Omega$ of $\Gamma$ is
$H_0$-hourglass-shaped around $x,y,z$ if, for any $w\in[x,z]$, the
ball $B(w, H_0+d(w,y)/2)$ is included in $\Omega$.
\end{definition}

The proof of Ancona inequalities from pre-Ancona inequalities (that
we described briefly after Lemma~\ref{lem:GL}) still works in
$H_0$-hourglass-shaped domains, since it shows that most trajectories
flow along the hourglass. This implies the following lemma (this is
Theorem~4.1 in~\cite{gouezel_lalley}).

\begin{lem}
\label{lem:ancona_hourglass}
Consider an admissible measure $\mu$ satisfying pre-Ancona
inequalities in a hyperbolic group. Let $H_0$ be large enough. There
exists $C>0$ such that, for any domain $\Omega$ which is
$H_0$-hourglass-shaped around three points $x,y,z$ on a geodesic (in
this order), the Green function relative to $\Omega$ satisfies Ancona
inequalities, that is,
\[
G(x,z; \Omega)\leq C G(x,y;\Omega) G(y,z;\Omega).
\]
\end{lem}

From this point on, we fix an admissible measure $\mu$ with
superexponential tails, which satisfies pre-Ancona inequalities. We
will prove that it satisfies strong Ancona inequalities. We fix the
constant $H_0$ given by Lemma~\ref{lem:ancona_hourglass} for this
measure.

The next lemma gives the basic inductive step for the proof of
Theorem~\ref{thmm:strong_ancona}. For $u,v,z\in\Gamma$, we write
$(u,v)_{z}$ for their Gromov product, given by $(u,v)_z =
(d(u,z)+d(v,z)-d(u,v))/2$. This is essentially the length of the part
that is common to two geodesics $[z,u]$ and $[z,v]$.

%
\includegraphics{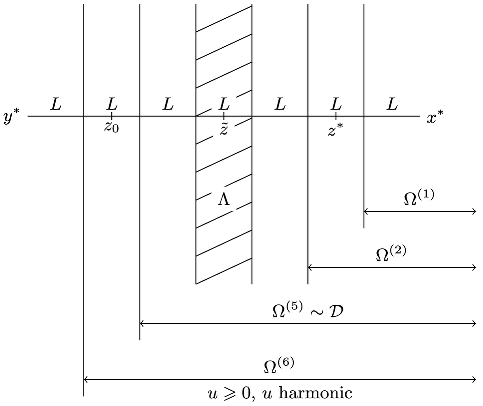}

%

\begin{figure}
\centering
\begin{tikzpicture}
\begin{scope}
 \clip
 (4, 2) -- (3, 2) -- (3,-3) -- (4, -3) -- cycle;
 \foreach\x in {-3, -2.5,..., 2}
 \draw(3,\x) -- ++(25:2cm);
 \end{scope}
\draw(3,0) -- node[above, fill=white]{$L$} node[below, fill=white]{$
\tilde z$} (4,0);
\node at (3.5, -1.7) [above, fill=white] {$\Lambda$};
\draw(0,0) node[left] {$y^*$} -- node[above] {$L$} (1,0)
 -- node[above] {$L$} node[below] {$z_0$} (2,0) -- node[above] {$L$}
(3,0)
 -- (4,0) -- node[above] {$L$} (5,0)
 -- node[above] {$L$} node[below]{$z^*$} (6,0) -- node[above]{$L$}
(7,0) node[right]{$x^*$};
\def\tickheight{0.07}
\draw(5.5, -\tickheight) -- (5.5, \tickheight);
\draw(5, -\tickheight) -- (5, \tickheight);
\draw(3.5, -\tickheight) -- (3.5, \tickheight);
\draw(1.5, -\tickheight) -- (1.5, \tickheight);
\draw(1, 2) -- (1,-5);
\draw[<->] (1, -4.7) -- node[above] {$\Omega^{(6)}$} node[below] {$u
\geq0$, $u$ harmonic} ++(7, 0);
\draw(2,2) -- (2, -4);
\draw[<->] (2, -3.7) -- node[above] {$\Omega^{(5)}\sim\boD$} ++ (6,0);
\draw(3,2) -- (3,-3);
\draw(4,2) -- (4, -3);
\draw(5,2) -- (5, -3);
\draw[<->] (5, -2.7) -- node[above] {$\Omega^{(2)}$} ++ (3,0);
\draw(6,2) -- (6, -2);
\draw[<->] (6, -1.7) -- node[above]{$\Omega^{(1)}$} ++(2, 0);
\end{tikzpicture}
\caption{The domains in Lemma~\protect\ref
{lem:strong_ancona}.}\label{fig1}
\end{figure}
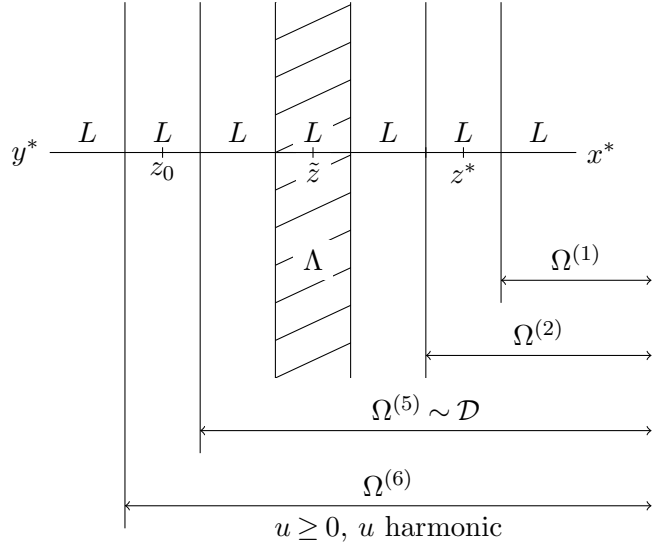

\begin{lem}
\label{lem:strong_ancona}
There exists $C_1>1$ such that, for any $D>0$, the following holds if
$L$ is a large enough even integer. Consider a geodesic segment
$\gamma$ between two points $x^*$ and $y^*$, of length $7L$. Let
$\Omega^{(j)}=\{z \st(y_*, z)_{x_*} \leq jL\}$ for $1\leq j\leq6$
(this is essentially the set of points whose projection on $\gamma$
is at distance at most $jL$ of $x^*$) and let $z^*$ be the point at
distance $3L/2$ of $x^*$ on $\gamma$. Let $\boH$ be the set of
functions $u\dvtx\Gamma\to\R$ satisfying the following properties:
\begin{longlist}[1.]
\item[1.] the function $u$ is positive on $\Omega^{(6)}$;
\item[2.] for all $z\in\Gamma$, one has $\vert u(z)\vert\leq D^{d(z,z^*)}
u(z^*) $;
\item[3.] the function $u$ is harmonic on $\Omega^{(6)}$, that is, $u(z)
= \sum_{w\in\Gamma} p(z,w)u(w)$ for all $z\in\Omega^{(6)}$
(note that the previous property ensures that this sum is
well defined, since $\mu$ has superexponential tails);
\item[4.] the function $\vert u(z)\vert$ is bounded by a finite linear
combination of functions $G(z, t_i)$.
\end{longlist}
Then there exists a domain $\boD$, included in $\Omega^{(6)}$ and
including $\Omega^{(5)}$ such that, for all $z\in\Omega^{(1)}$, for
all $u\in\boH$,
\[
C_1^{-1} \leq\frac{u(z)}{G(z,z^*; \boD)u(z^*)} \leq C_1.
\]
\end{lem}

Note that the Green function $G(z,z^*; \boD)$ satisfies a Harnack
inequality on $\Omega^{(1)}$, of the form $G(z,z^*;\boD) \leq
C_0^{d(z,z')} G(z',z^*;\boD)$ where the constant $C_0$ only depends
on $\mu$. Therefore, the conclusion of the lemma implies that, for
all $z,z'\in\Omega^{(1)}$, one has
\[
u(z)\leq C_1^2 C_0^{d(z,z')} u
\bigl(z'\bigr).
\]
This inequality should be compared to the second assumption on $u$,
involving an arbitrarily large constant $D$. Hence, the lemma asserts
that a weak growth control implies in fact a much stronger growth
control (but on a smaller domain). This remark will be crucial to
check inductively the assumptions of the lemma.

\begin{pf*}{Proof of Lemma \ref{lem:strong_ancona}}
Let $D>0$ be fixed, we will show the conclusion of the lemma if $L$
is large enough. We will write $o_L(1)$ for a term that may depend on
$D$ and $L$, and tends to $0$ when $L$ tends to infinity (with fixed
$D$). We will also write $C$ for generic constants that do not depend
on $D$. In particular, the constants in various Harnack inequalities
will be denoted by $C_0$.

\begin{step}
There exists a domain $\boD$, containing $\Omega^{(5)}$ and contained
in a fixed size neighborhood of $\Omega^{(5)}$, such that for all
$z,z'\in\boD$ there exists a path in $\boD$ from $z$ to $z'$ with
weight at least $C_0^{-d(z,z')}$.
\end{step}

\begin{pf}
The set $\Omega^{(5)}$ is convex, up to a constant $K_0$: any
geodesic between two points in $\Omega^{(5)}$ is contained in its
neighborhood $B(\Omega^{(5)}, K_0)$. Let $K_1$ be such that any
generator can be written as the product of at most $K_1$ elements in
the support of $\mu$. Between any points $z,z'\in\Omega^{(5)}$,
there exists a path staying in $B(\Omega^{(5)}, K_0+K_1)$ of length
at most $K_1 d(z,z')$ whose transitions are all in a finite subset of
the support of $\mu$. The weight of this path is therefore at least
$\bar C_0^{-d(z,z')}$, for some $\bar C_0>0$.

For all $z\in B(\Omega^{(5)}, K_0+K_1)$, choose a point $\zeta_z$ in
$\Omega^{(5)}$ with $d(z,\zeta_z)\leq K_0+K_1$, and choose two paths
$\tau_z$ and $\tau'_z$, respectively, from $z$ to $\zeta_z$ and from
$\zeta_z$ to $z$, with uniformly bounded length, and weight uniformly
bounded from below. Let finally $\boD$ be the union of all the
(points visited by the) paths $\tau_z$ and $\tau'_z$.

This set satisfies the required properties. Indeed, fix $w$ and
$w'\in\boD$, we construct a path from $w$ to $w'$ with weight at
least $C_0^{-d(w,w')}$ as follows. First, let $z$ be such that $w\in
\tau_z \cup\tau'_z$, and $z'$ be such that $w'\in\tau_{z'}\cup
\tau'_{z'}$. We can go from $w$ to $\zeta_z$ in $\tau_z \cup\tau'_z
\subset\boD$ with weight bounded from below, then from $\zeta_z$ to
$\zeta_{z'}$ in $B(\Omega^{(5)}, K_0+K_1)\subset\boD$ with weight at
least $\bar C_0^{-d(\zeta_z,\zeta_{z'})} \geq C^{-1} \bar
C_0^{-d(w,w')}$, and then from $\zeta_{z'}$ to $w'$ in $\tau_{z'}\cup
\tau'_{z'}\subset\boD$ with weight bounded from below. The
concatenation of these three paths stays in $\boD$ and has weight at
least $C_0^{-d(w,w')}$ for some $C_0$, as desired.
\end{pf}

We deduce in particular of the properties of $\boD$ that, for all $z,
z'\in\boD$,
%
\begin{equation}
\label{eq:compare_z_zprime} G\bigl(z,z^*; \boD\bigr) \geq C_0^{-d(z',z^*)} G
\bigl(z,z'; \boD\bigr)
\end{equation}
since a path from $z$ to $z'$ can be extended in $\boD$ by a path
from $z'$ to $z^*$ with weight at least $C_0^{-d(z',z^*)}$.

Let $u$ be a function in $\boH$.

\begin{step}\label{st3}
For all $z\in\Omega^{(2)}$,
%
\begin{equation}
\label{eq:u_step_1} u(z) = \sum_{w\in\Omega^{(6)}-\boD} G(z,w; \boD)u(w) +
o_L(1) G\bigl(z,z^*;\boD\bigr)u\bigl(z^*\bigr).
\end{equation}
\end{step}

One interest of this formula is that the values of $u$ appearing on
the right-hand side are all positive since $w\in\Omega^{(6)}$.

\begin{pf*}{Proof of Step \ref{st3}}
We start from $z$ and follow the random walk given by $\mu$ until
time $n$, stopping it when one exits $\boD$. Since $u$ is harmonic on
$\boD$, the average value of $u$ at time $n$ coincides with $u(z)$,
that is,
%
\begin{equation}
\label{eq:u_somme_n} u(z) = \sum_{w\notin\boD}
G_{\leq n}(z,w;\boD) u (w) + \sum_{w\in\boD}
p^{n}(z,w; \boD) u(w),
\end{equation}
where $G_{\leq n}(z,w;\boD)$ is the sum of the weights of all paths
from $z$ to $w$ of length at most $n$ that stay in $\boD$ except
maybe at the last step, and $p^{n}(z,w; \boD)$ is the same quantity
but for paths of length exactly $n$. Note that $G_{\leq n}(z,w;\boD)$
converges to $G(z,w;\boD)$ when $n$ tends to infinity.

By assumption, the function $\vert u\vert$ is bounded by a linear
combination of functions $G(z, t_i)$. For each of those functions,
$\sum_{w\in\Gamma}p^n(z,w) G(w,t_i)$ tends to $0$ when $n$ tends to
infinity (since this is the sum of the weights of paths from $z$ to
$t_i$ of length at least $n$). It follows that the last sum
in~\eqref{eq:u_somme_n} converges to $0$ with $n$. If $u$ were
positive, one would readily deduce that $u(z) = \sum_{w\notin\boD}
G(z,w;\boD) u (w)$ by passing to the limit. However, since $u$ can be
negative on the complement of $\Omega^{(6)}$, we should be more
careful. To justify the limit and equation~\eqref{eq:u_step_1}, it
suffices to show that
\[
\sum_{w\notin\Omega^{(6)}} G(z,w; \boD) \bigl\vert u(w)\bigr\vert \leq
o_L(1)G\bigl(z,z^*;\boD\bigr)u\bigl(z^*\bigr).
\]
Denoting by $z'$ the last point in $\boD$ of a trajectory from $z$ to
$w$, this sum can be written as
\[
\adjustlimits\sum_{w\notin\Omega^{(6)}} \sum
_{z'\in\boD} G\bigl(z,z'; \boD\bigr) p
\bigl(z', w\bigr)\bigl\vert u(w)\bigr\vert.
\]
Bounding $\vert u(w)\vert$ by $u(z^*)D^{d(w,z^*)}$ and using
inequality~\eqref{eq:compare_z_zprime}, we get that this is at most
\[
\adjustlimits\sum_{w\notin\Omega^{(6)}} \sum
_{z'\in\boD} G\bigl(z,z^*; \boD\bigr) C_0^{d(z',z^*)}
p\bigl(z', w\bigr)D^{d(w,z^*)} u\bigl(z^*\bigr).
\]
The required factor $G(z,z^*;\boD) u(z^*)$ can be factorized out, one
should show that the remaining term is $o_L(1)$. The measure $\mu$
has superexponential tails. Hence, for any $K$, one has $p(z',w) \leq
K^{-d(z',w)}$ if $L$ is large enough (since the jump from $z'$ to $w$
has size at least $L/2$). Hence, it suffices to show that
\[
\adjustlimits\sum_{w\notin\Omega^{(6)}} \sum
_{z'\in\boD} C_0^{d(z',z^*)}D^{d(w,z^*)}
K^{-d(z',w)} = o_L(1).
\]
Let $z_0$ be the point on $\gamma$ at distance $3L/2$ of $y^*$. By
hyperbolicity, any geodesic segment from $w$ to $z'$ passes within
bounded distance of $z_0$, and its length is at least $L/2$. Hence,
\begin{eqnarray*}
d\bigl(z', z^*\bigr) &\leq &d\bigl(z',
z_0\bigr) + d\bigl(z_0, z^*\bigr) \leq d
\bigl(z', w\bigr) + 7L \leq d\bigl(z',w\bigr) + 14 d
\bigl(z',w\bigr)\\
& = & 15 d\bigl(z',w\bigr).
\end{eqnarray*}
Moreover, $d(w,z^*) \leq d(w,z')+d(z',z^*) \leq16 d(z',w)$. Writing
$n=d(z',w)$, we deduce that the above sum is bounded by
\[
\sum_{n=L/2}^\infty\Card\bigl\{
\bigl(z'\in\boD,w\notin\Omega^{(6)}\bigr) \st d
\bigl(z',w\bigr) =n\bigr\} \bigl(C_0^{15}
D^{16} K^{-1}\bigr)^n.
\]
If $z'$ and $w$ are at distance $n$, they both belong to the ball
$B(z_0, n+C)$. Hence, $\Card\{(z',w) \st d(z',w) =n\}$ grows at most
exponentially fast, let us say that it is bounded by $C_2^n$. If $K$
was chosen so that $C_2 C_0^{15} D^{16} K^{-1}<1$, the above series
is converging, and can be made arbitrarily small by increasing $L$,
as desired.
\end{pf*}

\begin{step}
Define a domain $\Lambda= \Omega^{(4)}-\Omega^{(3)}$. For all $z\in
\Omega^{(2)}$,
%
\begin{eqnarray}
\label{eq:u_Lambda} u(z)& =& \adjustlimits\sum_{w\in\Omega^{(6)}-\boD}\sum
_{w'\in
\Lambda} G\bigl(z,w';\boD\bigr) G
\bigl(w',w;\boD-\Lambda\bigr) u(w)
\nonumber
\\[-8pt]
\\[-8pt]
\nonumber
&&\hspace*{64pt}{}+ o_L(1) G\bigl(z,z^*;\boD\bigr)u\bigl(z^*\bigr) +
o_L(1) u(z).
\end{eqnarray}
\end{step}

\begin{pf}
We start from expression~\eqref{eq:u_step_1}.
By~\eqref{eq:decompose_G}, every term $G(z,w;\boD)$ can be decomposed
as
\[
G(z,w;\boD) = \sum_{w'\in\Lambda} G\bigl(z,w';
\boD\bigr) G\bigl(w',w;\boD -\Lambda\bigr) + G(z,w;\boD-\Lambda),
\]
by considering the last visit of a trajectory to $\Lambda$ if it
exists. We have to show that the contribution of the terms
$G(z,w;\boD-\Lambda)$ is negligible. Let us consider a trajectory
$\tau$ from $z$ to $w$ that does not visit $\Lambda$, it has to jump
past $\Lambda$. Say that the first jump happens from a point $w_i$ to
a point $w_{i+1}$.

If $w_{i+1}=w$, that is, the trajectory has jumped directly out of
$\boD$, then we can use the same argument as in Step $2$ since we are
considering a trajectory ending with a very big jump. The same
argument shows that the overall contribution of those trajectories
to~\eqref{eq:u_step_1} is bounded by $o_L(1) G(z,z^*;\boD)u(z^*)$.

Assume now that $w_{i+1}\neq w$, and in particular $w_{i+1}\in\boD$.
Let $\tilde z$ be the middle point of $\Lambda$, located on $\gamma$
at distance $7L/2$ of $x^*$. As in Section~\ref{subsec:free}, we
define a modified trajectory $m(\tau)$ by removing the big jump, and
replacing it with two almost geodesic trajectories in $\boD$ from
$w_i$ to $\tilde z$ and from $\tilde z$ to $w_{i+1}$. The
construction of $\boD$ in Step $1$ ensures that one can find such
trajectories, with positive weight. One also adds loops around
$\tilde z$, counting the lengths of the trajectories from $w_i$ to
$\tilde z$ and from $\tilde z$ to $w_{i+1}$, to make sure that the
map $\tau\mapsto m(\tau)$ is one-to-one. As in
Section~\ref{subsec:free}, one verifies that the weight of
$m(\tau)$ is larger than the weight of $\tau$ [the ratio
$\pi(m(\tau))/\pi(\tau)$ even tends to infinity when $L$ tends to
infinity]. Summing over all those trajectories, we get that their
weight is bounded by $o_L(1) G(z,w;\boD)$.

It follows that the term we have to estimate, coming
from~\eqref{eq:u_step_1}, is bounded by
\[
o_L(1)\sum_{w\in\Omega^{(6)}-\boD} G(z,w;\boD) u(w).
\]
Formula~\eqref{eq:u_step_1} shows that the sum is bounded by $u(z) +
o_L(1) G(z,z^*;\boD)u(z^*)$. This completes the proof.
\end{pf}

In expression~\eqref{eq:u_Lambda}, we can bound each factor $G(z,
w';\boD)$ using Ancona inequalities in the hourglass-shaped domain
$\boD$ if $z\in\Omega^{(1)}$. Indeed, a geodesic from $z\in
\Omega^{(1)}$ to $w'\in\Lambda$ passes within bounded distance of
$z^*$ by hyperbolicity, and $\boD$ is $H_0$-hourglass-shaped around
$z,z^*,w'$ if $L$ is large enough. It follows from
Lemma~\ref{lem:ancona_hourglass} that $G(z,w';\boD)=C_3^{\pm1}
G(z,z^*;\boD) G(z^*, w';\boD)$ for some constant $C_3$ (this notation
means that the ratio of those quantities belongs to $[C_3^{-1},
C_3]$). As all the relevant values $u(w)$ are positive, we obtain
\begin{eqnarray*}
u(z)& =  & C_3^{\pm1} G\bigl(z,z^*;\boD\bigr) \adjustlimits
\sum_{w\in\Omega
^{(6)}-\boD} \sum_{w'\in\Lambda} G
\bigl(z^*,w';\boD\bigr) G\bigl(w',w;\boD-\Lambda\bigr)
u(w)
\\
&&\hspace*{138pt}{} + o_L(1) G\bigl(z,z^*;\boD\bigr)u\bigl(z^*\bigr) +
o_L(1) u(z).
\end{eqnarray*}
Applying again \eqref{eq:u_Lambda}, but to the point
$z^* \in\Omega^{(2)}$, we get that the double sum on the right-hand
side of the first line is equal to $u(z^*) + o_L(1) u(z^*)$. This
yields
\[
u(z) = C_3^{\pm1} G\bigl(z,z^*;\boD\bigr)u\bigl(z^*\bigr) +
o_L(1) G\bigl(z,z^*;\boD\bigr)u\bigl(z^*\bigr) + o_L(1)
u(z).
\]
Let $L$ be large enough so that the $o_L(1)$ terms are bounded by
$\min(C_3^{-1}/2, 1/2)$. We obtain that the ratio between $u(z)$ and
$G(z,z^*;\boD)u(z^*)$ is bounded from above and from below. This
completes the proof of Lemma~\ref{lem:strong_ancona}.
\end{pf*}

\begin{pf*}{Proof of Theorem~\ref{thmm:strong_ancona}}
Let us fix a large enough constant $D$ (several conditions will
appear in the proof below), and let $L$ be given for this value of
$D$ by Lemma~\ref{lem:strong_ancona}.

Starting with $4$ points $x,x',y,y'$ as in the statement of strong
Ancona inequalities, we want to show that~\eqref{eq:strong_ancona}
holds. Let $\tilde x$ and $\tilde y$ denote the branching points of
the tree between $\{x,x'\}$ and $\{y,y'\}$. We can without loss of
generality assume that $d(\tilde x, \tilde y)$ is of the form $7nL$
for some large integer $n$. We have to show that the functions
$u_0(z) = G(z,y)/G(\tilde x,y)$ and $v_0(z) = G(z,y')/G(\tilde x,
y')$ are exponentially close (in terms of $n$) in a domain containing
$x$ and $x'$.

Let $\gamma$ be a geodesic of length $7nL$ from $\tilde y$ to $\tilde
x$; we chop it into $n$ pieces $\gamma_i$ of length $7L$ (the piece
$\gamma_1$ is closest to $\tilde y$). We will successively apply
Lemma~\ref{lem:strong_ancona} along those pieces. We will denote by
$y_i^*$ and $x^*_i$ the endpoints of $\gamma_i$, by $z_i^*$ the point
at distance $3L/2$ of $x_i^*$ on $\gamma_i$, and by $\Omega_i^{(j)}$
the corresponding domains defined in Lemma~\ref{lem:strong_ancona}
for $1\leq j \leq6$.

Harnack inequalities show that $u_0$ satisfies $\vert
u_0(z)/u_0(z')\vert
\leq C_0^{d(z,z')}$ for some constant $C_0$. In particular, if $D\geq
C_0$, the function $u_0$ satisfies all the assumptions of
Lemma~\ref{lem:strong_ancona} along the geodesic $\gamma_1$. We
obtain a domain $\boD_1$ (that does not depend on $u_0$) such that
%
\begin{equation}
\label{eq:u0_compare} C_1^{-1} \leq\frac{u_0(z)}{G(z,z_1^*;\boD_1)u_0(z_1^*)} \leq
C_1,
\end{equation}
for all $z \in\Omega^{(1)}_1$. Using~\eqref{eq:u0_compare} at the
point $\tilde x$ and dividing, we get on $\Omega^{(1)}_1$
\[
C_1^{-2} \leq\frac{u_0(z)}{G(z,z_1^*;\boD_1) u_0(\tilde x)/G(\tilde
x,z_1^*;\boD_1)} \leq C_1^2.
\]
Let
\[
\phi_1(z) = \frac{1}{2C_1^2}\frac{G(z,z_1^*;\boD_1)}{G(\tilde
x,z_1^*;\boD_1)}u_0(
\tilde x).
\]
We note that $\phi_1$ depends on $u_0$ only through its value at
$\tilde x$. By construction, we have on $\Omega_1^{(1)}$
%
\begin{equation}
\label{eq:compare_phi1} \phi_1 \leq u_0/2 \leq
C_1^4 \phi_1.
\end{equation}
In particular, the function $u_1 = u_0-\phi_1$ is positive on
$\Omega_1^{(1)}$. It is also harmonic there. We will show that $u_1$
satisfies the assumptions of Lemma~\ref{lem:strong_ancona} with
respect to the geodesic segment $\gamma_2$. Since assumption~(4) is
trivial, we only have to prove the growth control~(2).

Let $z\in\Gamma$, we have to show that $\vert u_1(z)\vert \leq D^{d(z,
z_2^*)} u_1(z_2^*)$. We start with the case $z\in
\Omega_1^{(1)}-\{z_2^*\}$ (the case $z=z_2^*$ is trivial). By
construction, $u_1(z)\geq0$. Using (twice)~\eqref{eq:u0_compare},
and thanks to Harnack inequality, we get
\begin{eqnarray*}
\bigl\vert u_1(z)\bigr\vert & \leq &u_0(z) \leq C_1 G
\bigl(z,z_1^*;\boD_1\bigr)u_0
\bigl(z_1^*\bigr) \leq C_1 C_0^{d(z,z_2^*)}
G\bigl(z_2^*,z_1^*;\boD_1\bigr)
u_0\bigl(z_1^*\bigr)
\\
& \leq& C_1^2 C_0^{d(z,z_2^*)}
u_0\bigl(z_2^*\bigr) \leq2 C_1^2
C_0^{d(z,z_2^*)} u_1\bigl(z_2^*\bigr).
\end{eqnarray*}
If $D$ is large enough so that $2 C_1^2 C_0\leq D$, we obtain
$\vert u_1(z)\vert \leq D^{d(z,z_2^*)} u_1(z_2^*)$ for $z\in
\Omega_1^{(1)}-\{z_2^*\}$, as desired. Assume now that $z\notin
\Omega_1^{(1)}$. Thanks to Harnack inequalities,
\[
G\bigl(z,z_1^*;\boD_1\bigr) \leq G\bigl(z,z_1^*
\bigr)\leq C_0^{d(z,z_1^*)}G\bigl(z_1^*,
z_1^*\bigr) \leq C_2 C_0^{d(z,z_1^*)} G
\bigl(z_1^*, z_1^*;\boD_1\bigr)
\]
for some $C_2>0$. Hence, $\phi_1(z) \leq C_2 C_0^{d(z,z_1^*)}
\phi_1(z_1^*)$. As $\phi_1(z_1^*) \leq u_0(z_1^*)$
by~\eqref{eq:compare_phi1}, we obtain
\[
\bigl\vert u_1(z)\bigr\vert \leq\bigl\vert u_0(z)\bigr\vert +
\phi_1(z) \leq D^{d(z, z_1^*)} u_0
\bigl(z_1^*\bigr) + C_2 C_0^{d(z,z_1^*)}
u_0\bigl(z_1^*\bigr).
\]
If $D$ is large enough, this is bounded by $2D^{d(z,z_1^*)}
u_0(z_1^*)$. Inequality~\eqref{eq:u0_compare} at $z=z_2^*$,
combined with Harnack inequality, yields $u_0(z_1^*) \leq C_1
C_0^{d(z_1^*, z_2^*)} u_0(z_2^*)$. Since $u_0 \leq2 u_1$ on
$\Omega^{(1)}_1$, we obtain
\[
\bigl\vert u_1(z)\bigr\vert \leq4C_1 D^{d(z,z_1^*)}
C_0^{d(z_1^*, z_2^*)} u_1\bigl(z_2^*\bigr).
\]
As $z\notin\Omega_1^{(1)}$, we have $d(z,z_2^*) \geq d(z,z_1^*) +
L$, whereas $d(z_1^*, z_2^*)=7L$. Hence,
\[
\bigl\vert u_1(z)\bigr\vert \leq4C_1 \bigl(C_0^7
D^{-1}\bigr)^L D^{d(z,z_2^*)} u_1
\bigl(z_2^*\bigr).
\]
If $D$ is large enough so that $4C_1 C_0^7 D^{-1} \leq1$, we finally
get $\vert u_1(z)\vert \leq   D^{d(z, z_2^*)} u_1(z_2^*)$. This is the
requested inequality.

We have shown that the function $u_1$ satisfies the assumptions of
Lemma~\ref{lem:strong_ancona} along the geodesic segment $\gamma_2$.
Hence, we may apply the same argument: we obtain a function $\phi_2$
with $\phi_2 \leq u_1/2\leq C_1^4 \phi_2$ on $\Omega_1^{(2)}$, only
depending on $u_1$ through the value of $u_1(\tilde x)$ [and,
therefore, only depending on $u_0(\tilde x)]$. Let $u_2 = u_1-\phi_2$,
it again satisfies the assumptions of the lemma along $\gamma_3$, and
we can continue the construction inductively.

In the end, we construct $n$ functions $\phi_1,\ldots,\phi_n$ such
that $u_0 = u_n + \phi_1+\cdots+\phi_n$, only depending on
$u_0(\tilde x)$. As $u_k=u_{k-1}-\phi_{k} \leq(1-C_1^{-4}/2)
u_{k-1}$, we have in particular $u_n \leq(1-\epsilon)^n u_0$ on
$\Omega^{(1)}_n$, for $\epsilon=C_1^{-4}/2>0$. The same construction
can be done starting from the function $v_0(z)= G(z,y')/G(\tilde x,
y')$. Since $v_0(\tilde x)=u_0(\tilde x)=1$, the functions $\phi_i$
that we get are the same. Hence, on $\Omega^{(1)}_n$,
\[
\bigl\vert u_0(z) - v_0(z)\bigr\vert =\bigl \vert
u_n(z)- v_n(z)\bigr\vert \leq(1-\epsilon)^n
\bigl(u_0(z) + v_0(z)\bigr).
\]
Therefore,
\[
\bigl\vert u_0(z)/v_0(z)-1\bigr\vert \leq(1-
\epsilon)^n \bigl(u_0(z)/v_0(z) + 1\bigr).
\]
This implies that $u_0(z)/v_0(z)$ is bounded by
$(1+(1-\epsilon)^n)/(1-(1-\epsilon)^n) \leq2/\epsilon$, yielding
\[
\bigl\vert u_0(z)/v_0(z)-1\bigr\vert \leq C(1-
\epsilon)^n.
\]
In other words,
\[
\biggl\vert\frac{G(z, y)/G(\tilde x, y)}{G(z,y')/G(\tilde x, y')} - 1\biggr\vert\leq C(1-\epsilon)^n.
\]
Using this inequality at $z=x$ and $z=x'$ (those points belong to
$\Omega_n^{(1)}$), we get the conclusion of the theorem.
\end{pf*}


\begin{thebibliography}{12}

\bibitem{ancona2}
%
\begin{bincollection}[mr]
\bauthor{\bsnm{Ancona},~\bfnm{Alano}\binits{A.}}
(\byear{1988}).
\btitle{Positive harmonic functions and hyperbolicity}.
In \bbooktitle{Potential Theory---Surveys and Problems ({P}rague, 1987)}.
\bseries{Lecture Notes in Math.}
\bvolume{1344}
\bpages{1--23}.
\bpublisher{Springer},
\blocation{Berlin}.
\bid{doi={10.1007/BFb0103341}, mr={0973878}}
\end{bincollection}
%
\bptok{imsref}%
\endbibitem

\bibitem{anderson_schoen}
%
\begin{barticle}[mr]
\bauthor{\bsnm{Anderson},~\bfnm{Michael~T.}\binits{M.~T.}} \AND
\bauthor{\bsnm{Schoen},~\bfnm{Richard}\binits{R.}}
(\byear{1985}).
\btitle{Positive harmonic functions on complete manifolds of negative
curvature}.
\bjournal{Ann. of Math. (2)}
\bvolume{121}
\bpages{429--461}.
\bid{doi={10.2307/1971181}, issn={0003-486X}, mr={0794369}}
\end{barticle}
%
\bptok{imsref}%
\endbibitem

\bibitem{BHM_2}
%
\begin{barticle}[mr]
\bauthor{\bsnm{Blach{\`e}re},~\bfnm{S{\'e}bastien}\binits{S.}},
\bauthor{\bsnm{Ha{\"{\i}}ssinsky},~\bfnm{Peter}\binits{P.}} \AND
\bauthor{\bsnm{Mathieu},~\bfnm{Pierre}\binits{P.}}
(\byear{2011}).
\btitle{Harmonic measures versus quasiconformal measures for
hyperbolic groups}.
\bjournal{Ann. Sci. \'Ec. Norm. Sup\'er. (4)}
\bvolume{44}
\bpages{683--721}.
\bid{issn={0012-9593}, mr={2919980}}
\end{barticle}
%
\bptok{imsref}%
\endbibitem

\bibitem{bonk_schramm}
%
\begin{barticle}[mr]
\bauthor{\bsnm{Bonk},~\bfnm{M.}\binits{M.}} \AND
\bauthor{\bsnm{Schramm},~\bfnm{O.}\binits{O.}}
(\byear{2000}).
\btitle{Embeddings of {G}romov hyperbolic spaces}.
\bjournal{Geom. Funct. Anal.}
\bvolume{10}
\bpages{266--306}.
\bid{doi={10.1007/s000390050009}, issn={1016-443X}, mr={1771428}}
\end{barticle}
%
\bptok{imsref}%
\endbibitem

\bibitem{dynkin_martin}
%
\begin{barticle}[mr]
\bauthor{\bsnm{Dynkin},~\bfnm{E.~B.}\binits{E.~B.}}
(\byear{1969}).
\btitle{The boundary theory of {M}arkov processes (discrete case)}.
\bjournal{Uspehi Mat. Nauk}
\bvolume{24}
\bpages{3--42}.
\bid{issn={0042-1316}, mr={0245096}}
\end{barticle}
%
\bptok{imsref}%
\endbibitem

\bibitem{ghys_hyperbolique}
%
\begin{bbook}[mr]
\beditor{\bsnm{Ghys},~\bfnm{{\'E}.}\binits{{\'E}.}} \AND
\beditor{\bsnm{de la Harpe},~\bfnm{P.}\binits{P.}}, eds.
(\byear{1990}).
\btitle{Sur les Groupes Hyperboliques D'apr\`es {M}ikhael {G}romov}.
\bseries{Progress in Mathematics}
\bvolume{83}.
\bpublisher{Birkh\"auser},
\blocation{Boston, MA}.
\bid{doi={10.1007/978-1-4684-9167-8}, mr={1086648}}
\end{bbook}
%
\bptok{imsref}%
\endbibitem

\bibitem{gouezel_higherancona}
%
\begin{barticle}[mr]
\bauthor{\bsnm{Gou{\"e}zel},~\bfnm{S{\'e}bastien}\binits{S.}}
(\byear{2014}).
\btitle{Local limit theorem for symmetric random walks in
{G}romov-hyperbolic groups}.
\bjournal{J. Amer. Math. Soc.}
\bvolume{27}
\bpages{893--928}.
\bid{doi={10.1090/S0894-0347-2014-00788-8}, issn={0894-0347}, mr={3194496}}
\end{barticle}
%
\bptok{imsref}%
\endbibitem

\bibitem{gouezel_lalley}
%
\begin{barticle}[mr]
\bauthor{\bsnm{Gou{\"e}zel},~\bfnm{S{\'e}bastien}\binits{S.}} \AND
\bauthor{\bsnm{Lalley},~\bfnm{Steven~P.}\binits{S.~P.}}
(\byear{2013}).
\btitle{Random walks on co-compact {F}uchsian groups}.
\bjournal{Ann. Sci. \'Ec. Norm. Sup\'er. (4)}
\bvolume{46}
\bpages{129--173}.
\bid{issn={0012-9593}, mr={3087391}}
\end{barticle}
%
\bptok{imsref}%
\endbibitem

\bibitem{izumi_hyperbolic}
%
\begin{barticle}[mr]
\bauthor{\bsnm{Izumi},~\bfnm{Masaki}\binits{M.}},
\bauthor{\bsnm{Neshveyev},~\bfnm{Sergey}\binits{S.}} \AND
\bauthor{\bsnm{Okayasu},~\bfnm{Rui}\binits{R.}}
(\byear{2008}).
\btitle{The ratio set of the harmonic measure of a random walk on a
hyperbolic group}.
\bjournal{Israel J. Math.}
\bvolume{163}
\bpages{285--316}.
\bid{doi={10.1007/s11856-008-0013-6}, issn={0021-2172}, mr={2391133}}
\end{barticle}
%
\bptok{imsref}%
\endbibitem

\bibitem{ledrappier_freegroup}
%
\begin{bincollection}[mr]
\bauthor{\bsnm{Ledrappier},~\bfnm{Fran{\c{c}}ois}\binits{F.}}
(\byear{2001}).
\btitle{Some asymptotic properties of random walks on free groups}.
In \bbooktitle{Topics in Probability and {L}ie Groups: Boundary Theory}.
\bseries{CRM Proc. Lecture Notes}
\bvolume{28}
\bpages{117--152}.
\bpublisher{Amer. Math. Soc.},
\blocation{Providence, RI}.
\bid{mr={1832436}}
\end{bincollection}
%
\bptok{imsref}%
\endbibitem

\bibitem{sawyer_martin}
%
\begin{bincollection}[mr]
\bauthor{\bsnm{Sawyer},~\bfnm{Stanley~A.}\binits{S.~A.}}
(\byear{1997}).
\btitle{Martin boundaries and random walks}.
In \bbooktitle{Harmonic Functions on Trees and Buildings ({N}ew
{Y}ork, 1995)}.
\bseries{Contemp. Math.}
\bvolume{206}
\bpages{17--44}.
\bpublisher{Amer. Math. Soc.},
\blocation{Providence, RI}.
\bid{doi={10.1090/conm/206/02685}, mr={1463727}}
\end{bincollection}
%
\bptok{imsref}%
\endbibitem

\bibitem{woess}
%
\begin{bbook}[mr]
\bauthor{\bsnm{Woess},~\bfnm{Wolfgang}\binits{W.}}
(\byear{2000}).
\btitle{Random Walks on Infinite Graphs and Groups}.
\bseries{Cambridge Tracts in Mathematics}
\bvolume{138}.
\bpublisher{Cambridge Univ. Press},
\blocation{Cambridge}.
\bid{doi={10.1017/CBO9780511470967}, mr={1743100}}
\end{bbook}
%
\bptok{imsref}%
\endbibitem

\end{thebibliography}

%

%




\printaddresses
\end{document}